\documentclass{amsart}

\usepackage[mathcal]{euscript}
\usepackage{amssymb, amsfonts, xypic, amsmath}
\usepackage{amscd}

\newtheorem{thm}{Theorem}[section]
\newtheorem{lem}[thm]{Lemma}

\newtheorem{prop}[thm]{Proposition}

\theoremstyle{definition}
\newtheorem{defn}[thm]{Definition}

\newtheorem{ex}[thm]{Example}
\newtheorem{remark}[thm]{Remark}

\numberwithin{equation}{section} \numberwithin{figure}{section}

\numberwithin{equation}{section}

\newcommand{\R}{\mathbb{R}}

\newcommand{\Q}{\mathbb{Q}}

\newcommand{\dfn}{\textbf} 

\usepackage{graphicx}
\usepackage{epstopdf}

\title[Homotopy Algebra Structures on Twisted Tensor Products] {Homotopy Algebra Structures on Twisted Tensor Products and String Topology Operations}

\address{Department of Mathematics, The Graduate Center at CUNY, 365 Fifth Avenue, New York, NY 10016-4309}
\email{mmiller1@gc.cuny.edu}

\author{Micah Miller}

\begin{document}

\begin{abstract}
Given a $C_\infty$ coalgebra $C_*$, a strict dg Hopf algebra $H_*$, and
a twisting cochain $\tau:C_* \rightarrow H_*$ such that $Im(\tau)
\subset Prim(H_*)$, we describe a procedure for obtaining an
$A_\infty$ coalgebra on $C_* \otimes H_*$.  This is an
extension of Brown's work on twisted tensor products. We
apply this procedure to obtain an $A_\infty$ coalgebra model of the
chains on the free loop space $LM$ based on the $C_\infty$ coalgebra
structure of $H_*(M)$ induced by the diagonal map $M \rightarrow M
\times M$ and the Hopf algebra model of the based loop space given by $T(H_*(M)[-1])$.   When $C_*$ has cyclic $C_\infty$ coalgebra structure, we describe an $A_\infty$ algebra on $C_* \otimes H_*$.  This is used to give an explicit (non-minimal) $A_\infty$
algebra model of the string topology loop product.
Finally, we discuss a representation of the
loop product in principal $G$-bundles.

\end{abstract}
\maketitle
\tableofcontents

\section{Introduction}

Brown's theory of twisting cochains, outlined in \cite{B}, provides a way to model the total space of a bundle in terms of the base and fiber.  Given a principal bundle $G \rightarrow P \rightarrow M$ and a twisting cochain $\tau:C_*(M) \rightarrow C_*(G)$, Brown constructs a complex $(C_*(M) \otimes C_*(G), \partial_\tau)$ whose homology is isomorphic to $H_*(P)$.  If $Y$ is a $G$ space and $Y \rightarrow P \times_G Y \rightarrow M$ is the associated bundle, then there is a complex $(C_*(M) \otimes C_*(Y), \partial_\tau)$ whose homology is isomorphic to $H_*(P \times_G Y)$.  Quillen, in \cite{Q}, shows that when $Im(\tau)\subset Prim(H_*)$, the isomorphism is one of coalgebras.  There is an extensive literature on twisting cochains due to their wide ranging applications.  We have focused on these two results immediately related to this discussion.

In Section \ref{section:algebraicsetting}, we push Brown's theory to homotopy algebras.  That is, given a $C_\infty$ coalgebra $C_*$, a dg bialgebra $H_*$, and a twisting cochain $\tau:C_* \rightarrow H_*$ where $Im(\tau) \subset Prim(H_*)$, we define a twisted $A_\infty$ coalgebra on $C_* \otimes H_*$.   The twisted coalgebra structure is denoted $\{c_n^\tau:C_* \otimes H_* \rightarrow (C_* \otimes H_*)^{\otimes n} \}$.  The twisted term in Brown's differential is described by applying the coproduct on $C_*$, then applying $\tau$ to one of the factors, and finally using the multiplication in $H_*$.  The same idea is used for $c_1^\tau$, except we use the higher homotopies $\{c_n:C_* \rightarrow C_*^{\otimes n}\}$  of the $C_\infty$ coalgebra structure as well as the coproduct.  We use the same process to obtain $c_2^{\tau}$, except we use the maps $c_{n>2}$.  And the process continues for all $c_n^\tau$.  If $C_*$ is a strict differential graded coalgebra with $c_n=0$ for $n>2$, then the complex reduces to Brown's complex.  For this reason, we denote $c_1^\tau$ by $\partial_\tau$.  The following theorem is proved in Section \ref{section:algebraicsetting}.

\setcounter{thm}{8}
\setcounter{section}{3}
\begin{thm} 
Let $C_*$ be a $C_\infty$ coalgebra, $H_*$ a dg bialgebra, and $\tau:C_* \rightarrow H_*$ a twisting cochain such that $Im(\tau) \subset Prim(H_*)$.  The maps $\{\partial_\tau, c_2^\tau, c_3^\tau, \cdots \}$ define an $A_\infty$ coalgebra on $C_*  \otimes H_* $.
\end{thm}    

We then define the conjugation action of $H_*$ on itself.  The action of a primitive element on $H_*$ is both a derivation and a coderivation.  If we go through the process of defining $\{c_n^\tau\}$ as above, except instead of using the multiplication in $H_*$, we use the conjugation action, the resulting maps also define an $A_\infty$ coalgebra structure.  Because the conjugation action involves the antipode map, we require $H_*$ to be a dg Hopf algebra, as opposed to a dg bialgebra found in the Theorem \ref{thm:twistedcoalgebra}.  

\setcounter{thm}{16}
\begin{thm}
Let $C_*$ be a $C_\infty$ coalgebra, $H_*$ a dg Hopf algebra, and $\tau:C_* \rightarrow H_*$ a twisting cochain such that $Im(\tau) \subset Prim(H_*)$.  The maps $\{\partial_\tau, c_2^\tau, c_3^\tau, \cdots\}$, obtained using the conjugation action,  define an $A_\infty$ coalgebra on $C_*  \otimes H_* $.
\end{thm}

Since the conjugation action is a derivation, if $C_*$ also has a multiplication, it is reasonable to ask for an $A_\infty$ algebra on $C_* \otimes H_*$.  When $C_*$ is a cyclic $C_\infty$ coalgebra, there is a twisted $A_\infty$ algebra on $C_* \otimes H_*$. 
\setcounter{thm}{17}
\begin{thm} 
Let $C_*$ be a cyclic $C_\infty$ coalgebra,  $H_*$ be a Hopf algebra, and $\tau:C_* \rightarrow H_*$ be a twisting cochain with $Im(\tau) \subset Prim (H_*)$.   The maps $\{\partial_\tau, m_2, m_3, \cdots \},$ defined using the conjugation action in $H_*$, give $C_* \otimes_\tau H_*$ the structure of an $A_\infty$ algebra.
\end{thm}

The $A_\infty$ algebra and $A_\infty$ coalgebra share the same differential $\partial_\tau$, so they compute the same linear homology.  We still do not know what the further compatibilities are.

Section \ref{section:stringtopologyoperations} applies this work to the path space fibration $\Omega_b(M) \rightarrow P_b(M) \rightarrow M$.  Since $\Omega_b(M)$ is homotopy equivalent to a topological group, we consider $P_b(M) \rightarrow M$ to be a principal bundle.  The first step is to construct a twisting cochain $H_*(M) \rightarrow T(H_*(M)[-1])$, whose image is in $\mathcal{L}(H_*(M)[-1])$.    We obtain such a map by considering the construction of a power series connection.  Then we apply Theorem \ref{thm:twistedcoalgebra} to get an $A_\infty$ coalgebra model of the based path space. 

We also get a description of string topology operations from the path space fibration.  Any group acts on itself by conjugation. The conjugate bundle is defined to be the associated bundle of a principal $G$ bundle with respect to the conjugation action.  The conjugate bundle obtained from the path space fibration is a model of the free loop space.   Applying Theorem \ref{thm:conjugationcoalgebra} gives an $A_\infty$ coalgebra structure modeling the coalgebra on $H_*(LM)$ induced by the diagonal map.  Applying Theorem \ref{thm:loopproduct} gives an $A_\infty$ algebra structure modeling the algebra on $H_*(LM)$ given by the loop product. 

The final section applies the work in Section \ref{section:algebraicsetting} to the case of a principal $G$ bundle $G \rightarrow P \rightarrow M$, where $G$ is a connected Lie group.  The $A_\infty$ coalgebra on $H_*(M) \otimes H_*(G)$ given by applying Theorem \ref{thm:twistedcoalgebra} can be expressed in terms of the characteristic classes of the bundle.  We can also consider the conjugate bundle, denoted $Conj(P) \rightarrow M.$  Then Theorem \ref{thm:conjugationcoalgebra} gives an $A_\infty$ coalgebra model for $H_*(Conj(P))$ and Theorem \ref{thm:loopproduct}  gives an  $A_\infty$ algebra model.  

Given a connection on $P \rightarrow M$, there is a map of bundles $P_b(M) \rightarrow P $, which induces a map on associated bundles with respect to the conjugation action $Conj(P_b(M)) \rightarrow Conj(P)$.  Then the algebraic structures we get modeling the total space $Conj(P)$ are representations of algebraic structures on $Conj(P_b(M))$.  In this way, we get representations of string topology.

\thanks{\textbf{Acknowledgements:}}{ This paper would not have been possible without the help and direction of Mahmoud Zeinalian.  The author has also benefited from many helpful conversations with Joseph Hirsh, Jim Stasheff, David Stone, Dennis Sullivan, and Thomas Tradler.  The referee provided many necessary and constructive comments, which helped improve the paper. }

\setcounter{section}{1}

\section{Background Material}
Algebras and coalgebras are taken over $\Q$.  Homology and cohomology are taken with coefficients in $\Q$.  

\subsection{Twisting Cochains}

We first describe Brown's theory of twisting cochains in a purely algebraic setting.  Let $C_*$ be a differential graded coalgebra and $A_*$ a differential graded algebra.  Then $(Hom(C_*,A_*), \partial_C \otimes 1 + 1 \otimes \partial_A)$ is a differential graded algebra, and a twisting cochain is an element $\tau \in Hom(C_*,A_*)$ satisfying the Maurer Cartan equation $$\partial_A \circ \tau + \tau \circ \partial_C + \tau\cdot \tau=0. $$  The Maurer Cartan equation makes sense for any differential graded algebra, and a twisting cochain is a Maurer Cartan element in a differential graded algebra of the form $Hom(C_*,A_*)$.  The tensor differential $\partial_C \otimes1 + 1 \otimes \partial_A$ on $C_* \otimes A_*$ is twisted by adding a term
\begin{eqnarray*}
C_* \otimes A_* \overset{\Delta \otimes 1} {\rightarrow} C_* \otimes C_* \otimes A_* \overset{1 \otimes \tau \otimes 1}{\rightarrow} C_* \otimes A_* \otimes A_* \overset{1 \otimes m} {\rightarrow} C_* \otimes A_*.
\end{eqnarray*}
We refer to this term as the twisted term, and $\partial_\tau$ is the sum of the tensor differential and
twisted term.  The coproduct on $C_*$ defines a comodule on the tensor  $C_* \otimes A_* \rightarrow C_* \otimes C_* \otimes A_*$.  The coalgebra $C_*$ is a comodule over itself.

\begin{thm}\cite{B}
Let $C_*$ be a coalgebra, $A_*$ an algebra, and $\tau$ a twisting cochain.  Then $(C_* \otimes A_*, \partial_\tau)$ is a chain complex.  If $C_1=0$ and $\epsilon: A_* \rightarrow k$ is an augmentation, then $Id \otimes \epsilon: C_* \otimes A_* \rightarrow C_*$ is a map of comodules.
\end{thm}

\begin{proof}
In [\cite{B}, p. $229$],  $\partial_\tau$ is shown to square to zero.  We give a diagrammatic proof of that $\partial^2=0$ in Remark \ref{remark:pictureproof}.

The map $1 \otimes \epsilon$ obviously commutes with the comodule map, since the comodule map on $C_* \otimes A_*$ is given by the coproduct on $C_*$ and the coproduct on $C_*$ is the comodule structure for $C_*$.  To show that $1 \otimes \epsilon$ commutes with the differential, it suffices to show that $1 \otimes \epsilon$ vanishes on the twisted term.  To see this, note that $\epsilon$ is zero on any element of positive degree in $A_*$.  Let $c \otimes h \in C_* \otimes A_*$.  If $h$ is in positive degree, then the twisted term will have positive degree in the $A_*$ factor and will map to zero under $1 \otimes \epsilon$.  Consider $C_* \otimes1$ in $C_* \otimes A_*$ and $\Delta(c)= \sum c_{(1i)} \otimes c_{(2i)}$.   Since $\tau$ is a degree $-1$ map, $\tau(c_{(2i)})$ will have positive dimension for $|c_{(2i)}|>1$ and be zero for $|c_{(2i)}|=0$.  Since $C_1=0$, $1 \otimes \epsilon$ will vanish on the twisted term.

\end{proof}

We write $C_* \otimes_\tau A_*$ for the twisted complex $(C_* \otimes A_*, \partial_\tau)$.

This theory can be applied to principal bundles $G \rightarrow P \rightarrow M$.  The chain complex $C_*(M)$ is a differential graded coalgebra, where the coproduct is induced by the diagonal map $M \rightarrow M \times M$.  The group multiplication of $G$ provides an algebra structure on $C_*(G)$.  A twisting cochain is then a map $\tau:C_*(M) \rightarrow C_*(G)$ satisfying the Maurer Cartan equation.

The complex $(C_*(M) \otimes C_*(G), \partial_M \otimes Id + Id \otimes \partial_G)$ will not, in general, compute the homology of $P$.  However, when we twist the differential by a suitable twisting cochain $\tau:C_*(M) \rightarrow C_*(G)$, the complex $(C_*(M) \otimes C_*(G), \partial_\tau)$ will compute $H_*(P)$.

\begin{thm}[\cite{B}, Theorem $(4.2)$]
The chain complex $(C_*(M) \otimes C_*(G), \partial_\tau)$ is chain equivalent to $C_*(P)$.
\end{thm}

The equivalence of the above theorem is of chain complexes and not of dg coalgebras, despite the fact that both complexes have coproducts.  A further assumption is needed on $\tau$ to obtain an equivalence of dg coalgebras.

We return to the general setting.  Let $C_*$ be a dg coalgebra and $H_*$ a dg bialgebra.  The primitive elements $Prim(H_*)= \{h \in H_* | \Delta(h)= h \otimes 1 + 1 \otimes h \}$ is a Lie algebra whose universal enveloping algebra is $H_*$.  The following lemma is a reformulation of Quillen (\cite{Q}, Appendix B).

\begin{lem}
Let $\tau:C_* \rightarrow H_*$ be a twisting cochain from a cocommutative coalgebra to a dg bialgebra.  Then $(C_* \otimes H_*, \partial_\tau)$ is a differential graded coalgebra if and only if $Im(\tau) \subset Prim(H_*).$
\end{lem}
\begin{proof}

To show that $\partial_\tau$ is a coderivation we need to show that $$(\Delta_{C \otimes H}) \partial_\tau =(\partial_\tau \otimes 1 + 1 \otimes \partial_\tau)\Delta_{C \otimes H}. $$  The key is that multiplication by a primitive element is a coderivation.  We give a diagrammatic proof in Remark \ref{remark:pictureproof}.  The reader can find the computation in (\cite{Q}, p. $289$).

\end{proof}

\section{Algebraic Setting for Twisted Tensor Products}\label{section:algebraicsetting}

In this section, we extend the discussion of Brown's theory of twisting cochains to the homotopy algebra setting.  Let $(C_*, \{c_n\} )$ be a $C_\infty$ coalgebra and $H_*$ a strict dg bialgebra.  Given a twisting cochain $\tau:C_* \rightarrow H_*$, we define a twisted $A_\infty$ coalgebra structure on $C_* \otimes H_*$.

There are three properties that are used in Brown's setting.  For $C_*$ a strict dg coalgebra and $A_*$ a strict dg algebra, the following properties are used.

\begin{enumerate}
\item{$Hom(C_*,A_*)$ is a differential graded algebra. }
\item{twisting cochains $\tau:C_* \rightarrow A_*$ are in one to one correspondence with chain maps $\mathcal{F}(C_*) \rightarrow A_*$, where $\mathcal{F}$ is the cobar functor.}
\item{a twisting cochain $\tau \in Hom(C_*,A_*)$ defines a twisted differential on $C_* \otimes A_*$.}
\end{enumerate}
We address the analogs of these properties in the  following subsections.

\subsection{Maurer Cartan Equation in the Homotopy Algebra Setting}

We review some definitions.  An $A_\infty$ algebra consists of a vector space $V$ and maps $\{m_n:V[-1]^{\otimes n} \rightarrow V[-1]\}$ satisfying 
$$\sum_{k=1}^n \sum_{j=0}^{n-1} m_{n-k+1} \circ (Id^{\otimes j}\otimes  m_k \otimes Id^{n-j-k})=0. $$
The maps $\{m_n\}$ define a coderivation of square zero on $T(V[-1])$.  The shuffle product is a map $T(V[-1]) \otimes T(V[-1]) \rightarrow T(V[-1])$.  If $m_n$ vanishes on the image of the shuffle product for every $n$, then $(V, \{m_n\})$ is a $C_\infty$ algebra.  

An $A_\infty$ coalgebra and $C_\infty$ coalgebra are the dual notions of $A_\infty$ and $C_\infty$ algebras.  So $V$ is an $A_\infty$ coalgebra if there are maps $\{c_n:V[-1] \rightarrow V[-1]^{\otimes n}\}$ defining a derivation of square zero on $T(V[-1])$.  If the unshuffle product $T(V[-1]) \rightarrow T(V[-1]) \otimes T(V[-1])$ vanishes on the image of each $c_n$, then $(V, \{c_n\})$ is a $C_\infty$ coalgebra.

To deal with issues of convergence, we will make use of the completed tensor product.  For a vector space $V$, let $$\widehat T(V) = \prod_{i=0}^\infty V^{ \otimes i}. $$  In our applications, $V$ will be a finite dimensional vector space.  So $V$ has a unique topology making it a topological vector space.  There is an induced topology on $\widehat T(V)$, known the inverse limit topology.

 In order to say $\tau$ is a twisting cochain, the vector space $Hom(C_*, H_*)$ must have at least an $A_\infty$ algebra structure.  Moreover, we need the Lie algebra version of the Maurer Cartan equation, so we need an $L_\infty$ algebra on $Hom(C_*, Prim(H_*))$.

\begin{lem}\label{lem:Hom(C,A)}
Let $(C_*, \{c_n\})$ be a $C_\infty$ coalgebra and $A_*$ a differential graded algebra.  The vector space $Hom(C_*, A_*)$ is an $A_\infty$ algebra.
\end{lem}

Since $Hom(C_*, A_*) \cong C^* \otimes A_*$, the lemma is just the statement the tensor product of an $A_\infty$ algebra with an associative algebra is an $A_\infty$ algebra.  We omit the proof, but define the maps $m_n$.  Let \begin{eqnarray*}
m_1^{Hom}(f)= \partial_A \circ f + f \circ \partial_C,
\end{eqnarray*}
where $\partial_C= c_1$ of the $C_\infty$ coalgebra structure.  For $n>1$, let 
\begin{eqnarray*}
m_n^{Hom}(f_1, \cdots, f_n) :C_* &\rightarrow& A_*\\
c&\mapsto &   m_A (f_1 \otimes \cdots \otimes f_n))c_n(c),
\end{eqnarray*}
where by $m_A$ we mean multiply all the terms using multiplication of $A_*$.  Since the multiplication in $A_*$ is associative, $m_n^{Hom }$ is well-defined.

The Maurer Cartan equation is then
 $$\partial \circ \tau + \tau \circ \partial + m_2^{Hom}(\tau,\tau) + m_3^{Hom} (\tau,\tau,\tau) +  m_4^{Hom}(\tau,\tau,\tau,\tau) + \cdots =0. $$
Since we have an infinite sum, a note on convergence is in order.  In our application, $A_*= \widehat T(H_*(M)[-1]).$  The twisting cochain we construct will have the property that $$Im (m_n(\tau, \cdots, \tau)) \subset(H_*(M)[-1])^{\otimes n}$$  So the infinite sum can be expressed as a finite sum in different tensors.  This is well defined in the completed tensor product. 

For the Lie version of the Maurer Cartan equation, we will need the following fact about $L_\infty$ algebras.

The Koszul sign convention says that when two elements $x$ and $y$ of degree $p$ and $q$ are commuted, a sign of $(-1)^{pq}$ is obtained.  For $x_1, \cdots, x_n$ and a permutation $\sigma \in S_n$, let $\epsilon(\sigma; x_1, \cdots x_n)$ be the sign so that $$x_1 \wedge \cdots \wedge x_n = \epsilon(\sigma; x_1, \cdots, x_n) x_{\sigma(1)} \wedge \cdots \wedge x_{\sigma(n)},$$ in the free graded commutative algebra $\bigwedge(x_1, \cdots, x_n)$.  Let  $\xi(\sigma) = sgn(\sigma) \cdot \epsilon(\sigma; x_1, \cdots, x_n)$. 

\begin{thm}[\cite{LM}, Theorem $3.1$ ]
Let $(V, \{m_n\})$ be an $A_\infty$ algebra.  Then there is an $L_\infty$ algebra on $V$ given by symmetrizing $m_n$.  That is, if
\begin{eqnarray*}
l_n (v_1, \cdots, v_n) = \sum_{\sigma \in S_n} \xi(\sigma) m_n(v_{\sigma(1)} \otimes \cdots \otimes v_{\sigma(n)})
\end{eqnarray*}
then $(V, \{l_n\})$ is an $L_\infty$ algebra.

\end{thm}

We denote the $L_\infty$ algebra by $[V]$ to distinguish it from the $A_\infty$ algebra $V.$

\begin{lem}\label{lem:Hom(C,L)}
Let $(C_*, \{c_n\})$ be a $C_\infty$ coalgebra and $L_*$ be a differential graded Lie algebra.  Then $Hom(C_*, L_*)$ is an $L_\infty$ algebra.
\end{lem}

\begin{proof}

Our proof proceeds as follows.  Let $U(L_*)$ be the universal
enveloping algebra of $L_*$.  The previous lemma shows that the
space $Hom(C_*, U(L_*))$ is an $A_\infty$ algebra, with structure
maps $\{m_n\}$.  Symmetrizing each $m_n$ defines an $L_\infty$ algebra, with structure maps denoted $\{l_n\}$. Let $c_n(x) = x_{n,1} \otimes \cdots \otimes x_{n,n}.$
Then the $L_\infty$ algebra is given by
\begin{eqnarray*}
l_n(f_1 \cdots f_n) (x) = \sum_{\sigma \in S_n} \xi(\sigma) f_1(x_{n, \sigma(1)}) \cdots f_n(x_{n, \sigma(n)}).
\end{eqnarray*}
To prove the lemma, it suffices to show that the maps $\{l_n\}$ restricts to $Hom(C_*,L_*) \subset Hom(C_*,U(L_*)).$

Suppose $f_i \in Hom(C_*,L_*)$.  This implies $\Delta(f_i(x)) = f_i(x) \otimes 1 + 1 \otimes f_i(x), $ where the coproduct is in $U(L_*)$.  Since $\Delta$ is an algebra map, we see that
\begin{eqnarray*}
&\,&\Delta \circ l_n(f_1, \cdots, f_n) (x) \\
&=&  \sum_{\sigma \in S_n} \Delta (f_1(x_{n, \sigma(1)}))   \cdots \Delta(f_n(x_{n, \sigma(n)})) \\
&=&  \sum_{\sigma \in S_n} (f_1(x_{n, \sigma(1)} \otimes 1 + 1 \otimes f_1(x_{n,\sigma(1)}))) \cdots  (f_n(x_{n, \sigma(n)}) \otimes 1 + 1 \otimes f_n(x_{n, \sigma(n)})) \\
&=& \sum_{\sigma \in S_n}  f_1(x_{n, \sigma(1)}) \cdots f_n(x_{n, \sigma(n)}) \otimes 1 + 1 \otimes f_1(x_{n, \sigma(1)}) \cdots f_n(x_{n, \sigma(n)})
 \\
&\,&+ \sum_{\sigma \in S_n} \sum_j f_1(x_{n, \sigma(1) }) \cdots f(x_{n, \sigma(j)}) \otimes f(x_{n, \sigma(j+1)}) \cdots f_n(x_{(n, \sigma(n))}).
\end{eqnarray*}
We need to show that the cross terms cancel.  The composition $$C_* \overset{c_n}{\rightarrow} C_*^{\otimes n} \hookrightarrow T(C_*) \overset{\text{unshuffle}}{\rightarrow} T(C) \otimes T(C) $$ is zero by definition of a $C_\infty$ coalgebra.  Each permutation $\sigma$ is an $(i,j)$ unshuffle of some linear order of the $\{x_{n,i} \}$.  For example, for $S_3$, the collection of all the $(2,1)$ unshufflings of $x_{3,1} \otimes  x_{3,2} \otimes x_{3,3}$ and $x'_{3,1} \otimes x'_{3,2} \otimes x'_{3,3}= x_{3,2}\otimes  x_{3,1}\otimes  x_{3,3}$ exhausts all combinations of $x_{3,\sigma(1)} \otimes  x_{3,\sigma(2)} \otimes x_{3,\sigma(3)}$.

 The $L_\infty$ algebra on $Hom(C_*,L_*)$ is then given by
 $$l_n(f_1, \cdots f_n) (x) = \sum_{\sigma \in S_n} \xi(\sigma) f(x_{1,\sigma(1)}) \cdots f(x_{n,\sigma(n)}), $$
where the multiplications are in $U(L_*)$.
\end{proof}

Let $A_*$ and $B_*$ be two $A_\infty$ algebras and $\{f_n:A_*^{\otimes n} \rightarrow B_* \}$ an $A_\infty$ algebra morphism.  Suppose the Maurer Cartan equation is well defined for $A_*$ and $B_*$ (so either there are only finitely many maps defining the $A_\infty$ algebra or a suitable notion of convergence of the infinite sum holds). Let $\tau \in A_*$ be a Maurer Cartan element.  That is,
$$\partial_A \tau + m_2^A(\tau \otimes \tau) + m_3^A(\tau \otimes \tau \otimes \tau) + \cdots=0.$$
The following well-known lemma shows how to obtain a Maurer Cartan element in $B_*$ from $\tau$ and $\{f_n\}$.

\begin{lem}
Let $A_*,B_*$ be two $A_\infty$ algebras and $\{f_n:A_*^{\otimes n}
\rightarrow B_*\}$ be an $A_\infty$ algebra morphism between them.  If $\tau$ is a Maurer
Cartan element in $A_*$ then
$$\tau' = f(\tau) + f_2(\tau \otimes \tau) + \cdots + f_n(\tau^{\otimes n}) + \cdots $$
is a Maurer Cartan element in $B_*$.
\end{lem}

\subsection{Maurer Cartan Equation and Differential Graded Algebra Maps}

The following lemmas will be used to construct twisting cochains.

\begin{lem}
Let $C_*$ be an $A_\infty$ coalgebra and $A_*$ an associative algebra.  There is a one to one correspondence between twisting cochains $\tau:C_* \rightarrow A_*$ and differential graded algebra maps $\tau_T: T(C_*[-1]) \rightarrow A_*$. \end{lem}

\begin{proof}
Let $\partial^{T(C)}:T(C_*[-1]) \rightarrow T(C_*[-1])$ be the derivation of
square zero given by the $A_\infty$ coalgebra on $C_*$.
Given a twisting cochain $\tau:C_* \rightarrow A_*$, let $\tau_T(c_1
\otimes \cdots \otimes c_n) = \tau(c_1) \cdots \tau(c_n).$  Then by
construction, $\tau_T$ is an algebra map. It is a chain map, because
\begin{eqnarray*}
\partial^H (\tau (c)) &=& \tau \partial^C (c) + m_2^{A} \circ (\tau \otimes \tau) \circ c_2( c) + m_3^A \circ \tau^{\otimes 3}\circ  c_3(c) \\
&=& \tau \partial^{T(C)}(c),
\end{eqnarray*}
where the first equality is due to the Maurer Cartan equation for $\tau$ and the second equality is the definition of $\partial^{T(C)}$ in terms of the maps $c_n:C_*[-1] \rightarrow C_*[-1]^{\otimes n}$.  Conversely, given a map of differential graded algebras $\tau_T:T(C_*)\rightarrow A_*$ restricting $\tau$ to $C_*$ defines a twisting cochain.

\end{proof}

\begin{lem} \label{lem:liemapsandMC}
Let $C_*$ be a $C_\infty$ coalgebra and $H_*$ a Hopf algebra.   There is a one to one correspondence between twisting cochains $\tau:C_* \rightarrow Prim(H_*)$ and differential graded Lie algebra maps $\mathcal{L}(C_*[-1]) \rightarrow Prim(H_*)$.
\end{lem}

\begin{proof}

This lemma is proved in the same way as that of the previous.  Note that a $C_\infty$ coalgebra defines a derivation of square zero on the free Lie algebra $\mathcal{L}(C_*[-1])$.

\end{proof}

\subsection{$C_\infty \, \text{coalg} \, \otimes_\tau \, \text{bialg}$ as an $A_\infty$ coalgebra using left multiplication} \label{subsection:multiplication}

Given a twisting cochain $\tau:C_* \rightarrow H_*$, we want to define a twisted $A_\infty$ coalgebra structure on $C_* \otimes H_*$.  First, we define the untwisted $A_\infty$ coalgebra.

\begin{lem}
Let $(C_*, \{c_n\})$ be an $A_\infty$ coalgebra and $H_*$ be an algebra with a strictly coassociative comultiplication.  Then $C_* \otimes H_*$ is an $A_\infty$ coalgebra with structure maps 
\begin{eqnarray*}
c_n^\otimes = c_n \otimes \left( (\Delta \otimes Id^{\otimes n-1}) \circ \cdots \circ \Delta \right) :C_* \otimes H_* \rightarrow (C_*\otimes H_*)^{\otimes n}.
\end{eqnarray*}
\end{lem}
\begin{proof}
The proof is straightforward, using the $A_\infty$ coalgebra relations for $C_*$ terms and that $H_*$ is a strict coassociative coalgebra.
\end{proof}

\begin{remark}\label{remark:pictureproof}
Before we define an $A_\infty$ coalgebra structure on $C_* \otimes H_*$, we return to the classical setting of Brown's twisting cochains.  We introduce a graphical picture of $\partial_\tau$ and a graphical proof that $\partial_\tau^2=0$.  This technique will be used to define the twisted $A_\infty$ coalgebra later on.  Let $C_*$ be a differential graded coalgebra and $H_*$ a differential graded bialgebra.  Let $\tau$ be a twisting cochain and $\partial_\tau$ be the twisted differential.

To represent $\partial_\tau:C_* \otimes H_* \rightarrow C_* \otimes H_*$, we draw two vertical lines, one to represent $C_*$ the other to represent $H_*$.  We draw a horizontal dash to denote the differential.  The twisting term applies the coproduct on $C_*$ and $\tau$ to one of the factors.  We represent the twisting cochain $\tau:C_* \rightarrow H_*$ by connecting the lines representing $C_*$ and $H_*$ with a line.  The resulting vertex on $C_*$ of valence three can be thought of as the coproduct and the vertex of valence three on $H_*$ can be thought of as the product.  We refer the reader to Figure \ref{figure:twisteddifferential}   for a picture of $\partial_\tau$.

\begin{figure}
\includegraphics[width=2.75in]{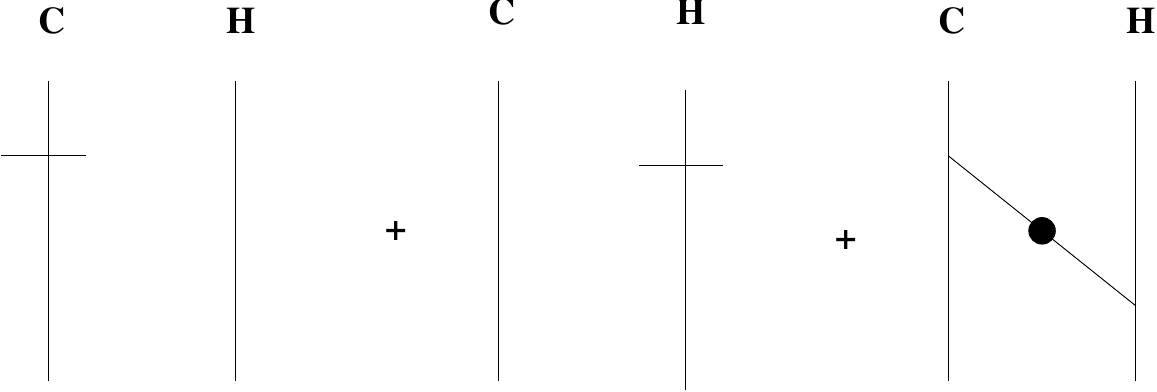}  \\

\caption{A graphical representation of $\partial_\tau= \partial_C \otimes1 + 1 \otimes \partial_A + (1\otimes m_A \otimes\tau \otimes 1) \Delta_C \otimes1.$  A vertical line with a dash represents the differential.  The diagonal line with a vertex represents the map $\tau:C \rightarrow H$.}
 \label{figure:twisteddifferential}
\end{figure}

We can prove that $\partial_\tau^2=0$ by analyzing the diagrams.  The top row in Figure \ref{figure:d2=0}  are the terms that remain after canceling the terms in $\partial_\tau^2$ that correspond to the tensor differential, which is well known to square to zero.  Note that because $\partial_C$ is a coderivation, the first and third terms in this row are equal to the first term in the second row of the figure.  Similarly, since $\partial_H$ is a derivation, the second and fourth terms on the first row equal the second term in the second row.  The coassociativity of $\Delta_C$ and the associativity of $m_H$ imply the last term of the first row is equal to the last term of the second row.  The bottom row then is equal to zero, because the middle lines describe the Maurer Cartan equation $\partial_H \tau + \tau \partial_C  + \tau \cdot \tau,$ which is zero by assumption.

\begin{figure}
\includegraphics[width=4.5in]{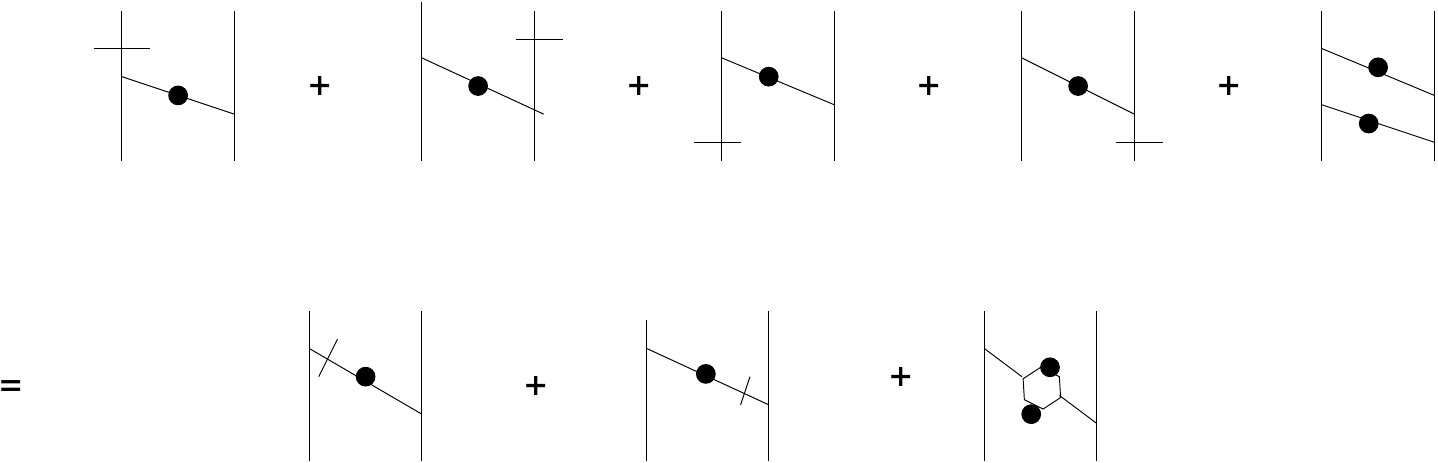}  \\

\caption{A graphical representation of $\partial_\tau^2=0$.  The top row represents the five terms that remain in $\partial_\tau^2$ when we cancel the terms corresponding to the tensor differential.  The bottom row is zero because the middle lines represent $\partial_H \tau + \tau \partial_C + \tau \cdot \tau$. }
\label{figure:d2=0}
\end{figure}

There is a similar argument showing that if $Im(\tau) \subset Prim(H_*)$, then $(C_* \otimes H_*, \partial_\tau)$ is a differential graded coalgebra.  The argument requires $C_*$ to be a cocommutative coalgebra.  We refer the reader to Figure \ref{figure:strictcoderivation}.

\begin{figure}
\includegraphics[width=3in]{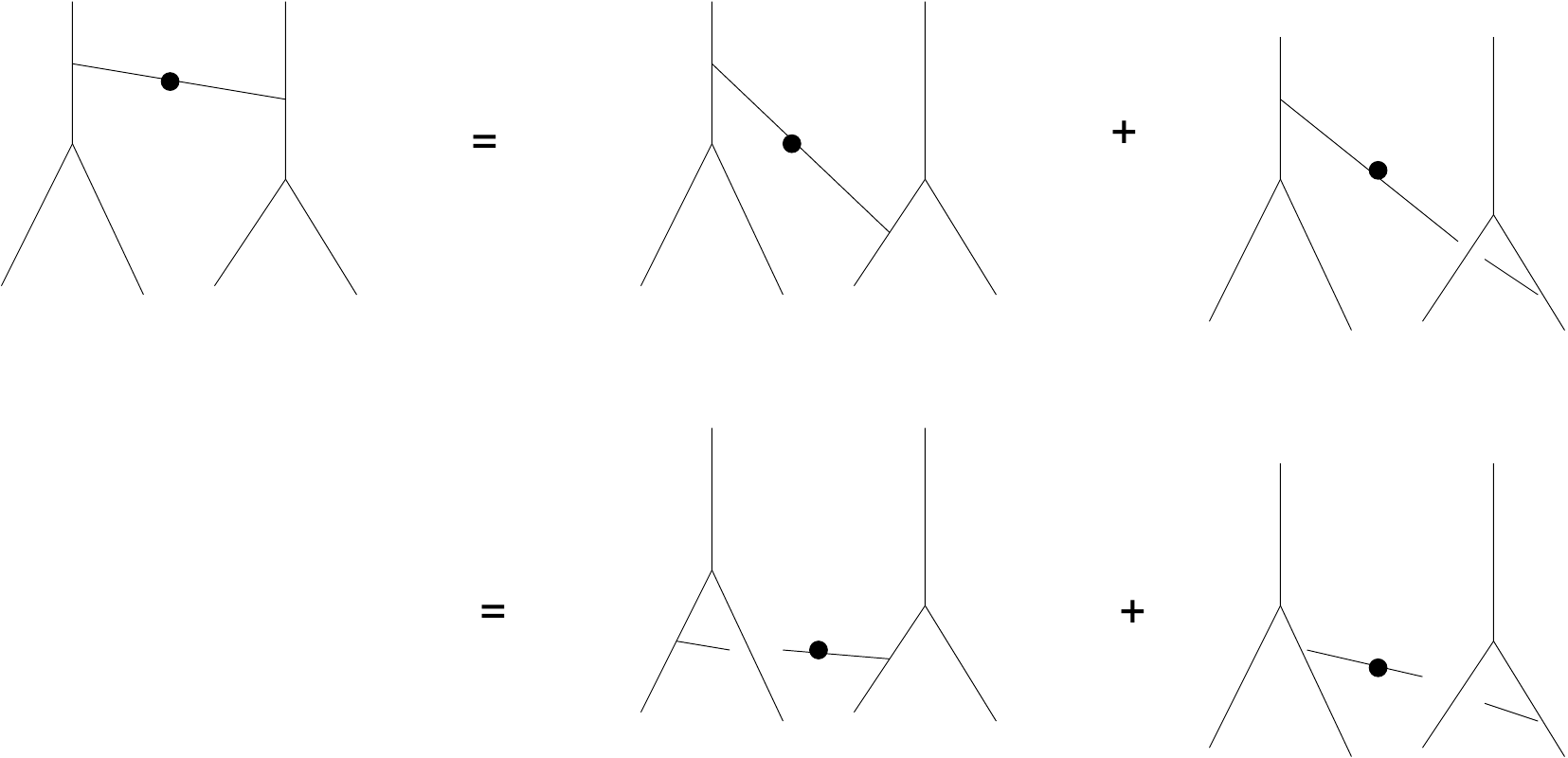}  \\

\caption{A graphical representation that $\partial_\tau$ is  a coderivation of the coproduct of $C_* \otimes H_*$.  The first equality is a result of the fact that multiplication by a primitive element is a coderivation.  The second equality is a result of the coproduct in $C_*$ being coassociative and cocommutative. }
\label{figure:strictcoderivation}
\end{figure}

\end{remark}

We can now describe how to twist the $A_\infty$ coalgebra.  Let $\tau:C_* \rightarrow Prim(H_*)$ satisfy the Lie Maurer Cartan equation.  First consider $c_1^{Hom}:C_* \otimes H_* \rightarrow C_* \otimes H_*$.  As in the strict setting, there is a twisting term of the form
\begin{eqnarray*}
C_* \otimes H_* &\overset{c_2}{\rightarrow}& C_* ^{\otimes 2} \otimes  H_* \overset{1 \otimes \tau \otimes 1}{\rightarrow}   C_* \otimes H_* ^{\otimes 2}\overset{1 \otimes m_H}{ \rightarrow} C_* \otimes H_*.
\end{eqnarray*}
But this twisting only takes $c_2$ into account and ignores all of the higher $c_n$ maps in the $C_\infty$ coalgebra structure on $C_*$.  To account for these maps, first apply $c_n$ to $C_*$ and apply $\tau^{\otimes n-1}$ to the last $n-1$ factors in $C_*^{\otimes n}$.  Since $Im(\tau) \subset Prim(H_*)$, we can bracket these $n-1$ terms in all possible ways to get another primitive element.  Then we multiply $Prim(H_*)$ and $H_*$ terms.  To sum up, $c_1^\tau$ consists of terms
\begin{eqnarray*}
C_* \otimes H_* & \overset{c_3 \otimes 1} {\rightarrow}& C_*^{\otimes 3} \otimes H_* \overset{1 \otimes \tau^{\otimes 2} \otimes 1} {\rightarrow} C_* \otimes H_*^{\otimes 2} \otimes H \overset{1 \otimes [,] \otimes 1}{\rightarrow} C_* \otimes H_* \otimes H_* \overset{1 \otimes m}\rightarrow C_* \otimes H_* \\
C_* \otimes H_* & \overset{c_4 \otimes 1} {\rightarrow}& C_*^{\otimes 4} \otimes H_* \overset{1 \otimes \tau^{\otimes 3} \otimes 1} {\rightarrow} C_* \otimes H_*^{\otimes 3} \otimes H \overset{1 \otimes [,] \otimes 1}{\rightarrow} C_* \otimes H_* \otimes H_* \overset{1 \otimes m}\rightarrow C_* \otimes H_* \\
C_* \otimes H_* & \overset{c_5 \otimes 1} {\rightarrow}& C_*^{\otimes 5} \otimes H_* \overset{1 \otimes \tau^{\otimes 4} \otimes 1} {\rightarrow} C_* \otimes H_*^{\otimes 4} \otimes H \overset{1 \otimes [,] \otimes 1}{\rightarrow} C_* \otimes H_* \otimes H_* \overset{1 \otimes m}\rightarrow C_* \otimes H_*
\end{eqnarray*}
and continue for all $n$ in this way.  By $[,]$ for three or more terms, we mean $$[x_1, \cdots, x_n] = \sum_{\sigma \in S_n}[x_{\sigma(1)}, [x_{\sigma(2)} , \cdots [x_{\sigma(n-1)}, x_{\sigma(n)}]]] .$$  Note the similarity of the twisted terms with the $L_\infty$ algebra on $Hom(C_*, Prim(H_*))$.  Since $c_1^\tau$ is an infinite sum, we need to address the issue of convergence  in $C_* \otimes H_*$.  In our application, $H_* = \widehat T(H_*(M)[-1])$, with the multiplication given by concatenation of tensors.  Let $x\in C_* \otimes H_*(M)[-1]$.  When $c_n$ is used to twist the differential, the corresponding term in $c_1^\tau(x)$ will be an element in $C_* \otimes (H_*(M)[-1])^{\otimes n}.$  Then $c_1^\tau$ consists of finite sums in different tensor products.  So in the completed tensor product, $c_1^\tau(x)$ is well defined.    

When $C_*$ is a strict dg coalgebra, then $c_1^\tau$ is the same as the twisted differential $\partial_\tau$ in Brown's construction.  So we write $c_1^\tau$ by $\partial_\tau$.

The higher maps $c_n$ can be twisted in the same manner as $c_1$.  To twist $c_2:C_* \otimes H_* \rightarrow C_*^{\otimes 2} \otimes H_*^{\otimes 2}$, we apply $c_n$ for $n>2$, then $\tau^{n-1}$ to the last $n-2$ factors of $C_*^{\otimes n}$, and bracketing these $n-2$ terms in all possible ways, multiplying the result with the element in $H_*$, and finally applying the coproduct in $H_*$.  For $n=3$, the process is the composition of
\begin{eqnarray*}
C_* \otimes H_* \overset{c_3 \otimes 1} {\rightarrow} C_*^{\otimes 3}\otimes H_* \overset{1^{\otimes 2} \otimes \tau \otimes 1} {\rightarrow} C_*^{\otimes 2} \otimes H_* \otimes H_*  \overset{1 \otimes m} {\rightarrow} C_*^{\otimes 2} \otimes H_*  \overset{1^{\otimes 2} \otimes \Delta}\rightarrow C_*^{\otimes 2} \otimes H_*^{\otimes 2}.
\end{eqnarray*}
The resulting map is denoted $c_2^\tau:C_* \otimes H_* \rightarrow (C_*  \otimes H_*)^{\otimes 2}$.

For $n>3$, we must use the Lie bracket, and the composition of maps is
\begin{eqnarray*}
C_* \otimes H_*  \overset{c_n \otimes 1} {\rightarrow} C_*^{\otimes n}\,\otimes H_* \overset{1^{\otimes 2} \otimes \tau^{\otimes n-2} \otimes 1} {\rightarrow} C_*^{\otimes 2} \otimes H^{\otimes n-2} \otimes H_*  \\
\overset{1^{\otimes 2} \otimes [,]^{n-2} \otimes 1} {\rightarrow} 1 \otimes H_* \otimes H_*
 \overset{1 \otimes m} {\rightarrow} C_*^{\otimes 2} \otimes H_*  \overset{1^{\otimes 2} \otimes \Delta}\rightarrow C_*^{\otimes 2} \otimes H_*^{\otimes 2}.
\end{eqnarray*}
To show that $\{c_n^\tau\}$ defines an $A_\infty$ coalgebra on
$C_* \otimes H_*$, we use the diagrams as in Remark \ref{remark:pictureproof}.  For a
picture of $\partial_\tau$ we refer the reader to Figure
\ref{figure:twistedd2}.  For a picture of $c_2^\tau$, we refer the
reader to Figure \ref{figure:twistedc2}.  Since multiplying by a
primitive element is a coderivation, we have some identities for the
terms in $c_2^\tau$.  These identities are described in Figure
\ref{figure:primitivecoderivation}.

\begin{figure}
\includegraphics[width=4.5in]{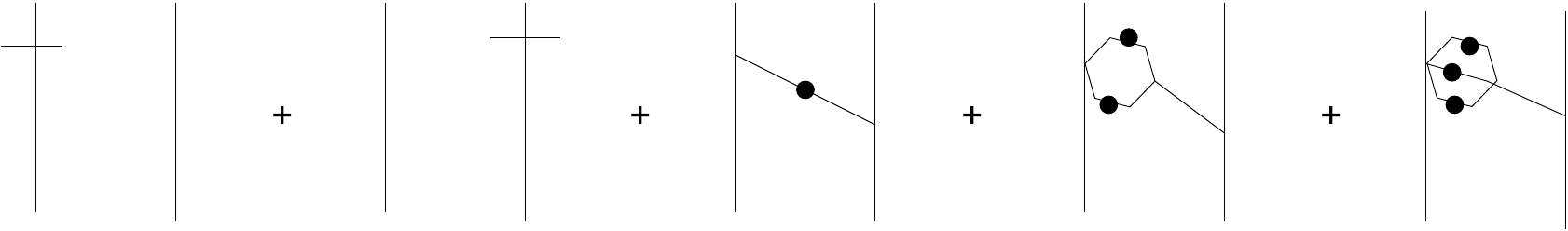}  \\

\caption{A graphical representation of $\partial_\tau$.  The terms are $\partial_C \otimes1 + 1 \otimes \partial_H +   (1 \otimes m) (1 \otimes \tau \otimes 1)c_2 \otimes 1 +  (1 \otimes 1 \otimes \tau )c_3$ }
 \label{figure:twistedd2}
\end{figure}

\begin{figure}
\includegraphics[width=4.5in]{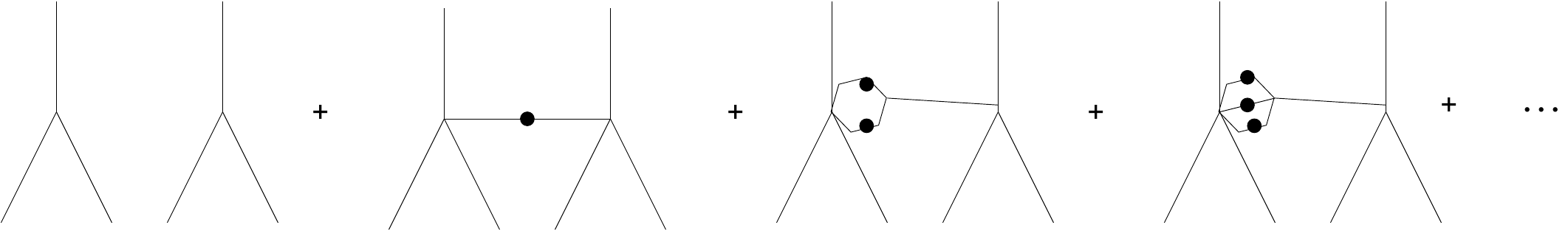}  \\

\caption{A graphical representation of $c_2^\tau$.}
\label{figure:twistedc2}
\end{figure}

\begin{figure}
\includegraphics[width=3.5in]{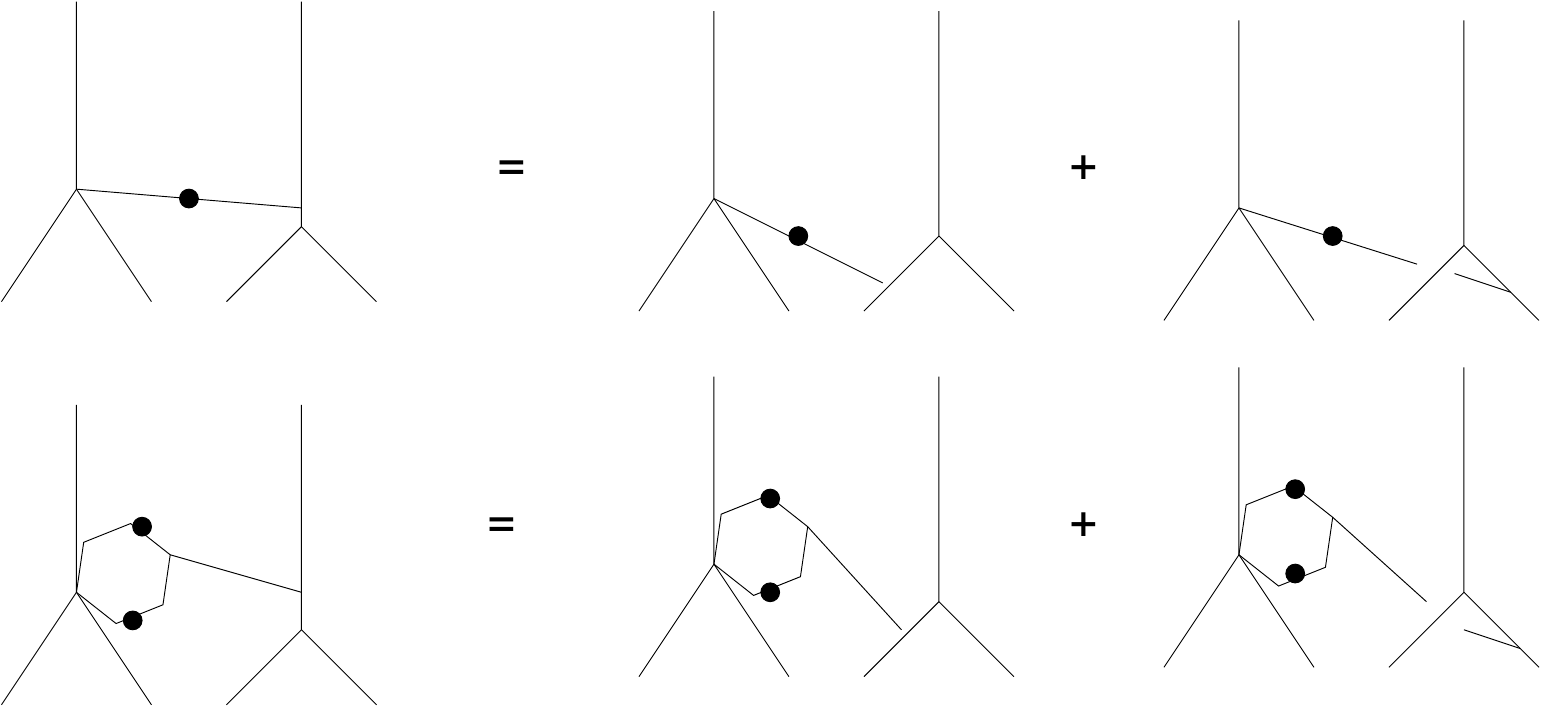}  \\

\caption{The above identities hold because $Im(\tau) \subset
Prim(H_*)$ and multiplying by a primitive element is a coderivation.
The same holds true for the other terms of $c_2^\tau$ and also for
$c_n^\tau$.} \label{figure:primitivecoderivation}
\end{figure}

We can now show that $\{\partial_\tau, c_2^\tau, c_3^\tau, \cdots \}$ define an $A_\infty$ coalgebra.  The proof of the theorem uses a graphical approach.  

\begin{thm} \label{thm:twistedcoalgebra}
Let $C_*$ be a $C_\infty$ coalgebra, $H_*$ a dg bialgebra, and $\tau:C_* \rightarrow H_*$ a twisting cochain such that $Im(\tau) \subset Prim(H_*)$.  The maps $\{\partial_\tau, c_2^\tau, c_3^\tau, \cdots \}$ define an $A_\infty$ coalgebra on $C_*  \otimes H_* $.
\end{thm}

\begin{proof}
We first show that $\partial_\tau$ is a differential.  To show that $\partial_\tau^2=0$ we will show that  expanding the terms yield many occurrences of the Maurer Cartan equation.

\begin{figure}
\includegraphics[width=4.5in]{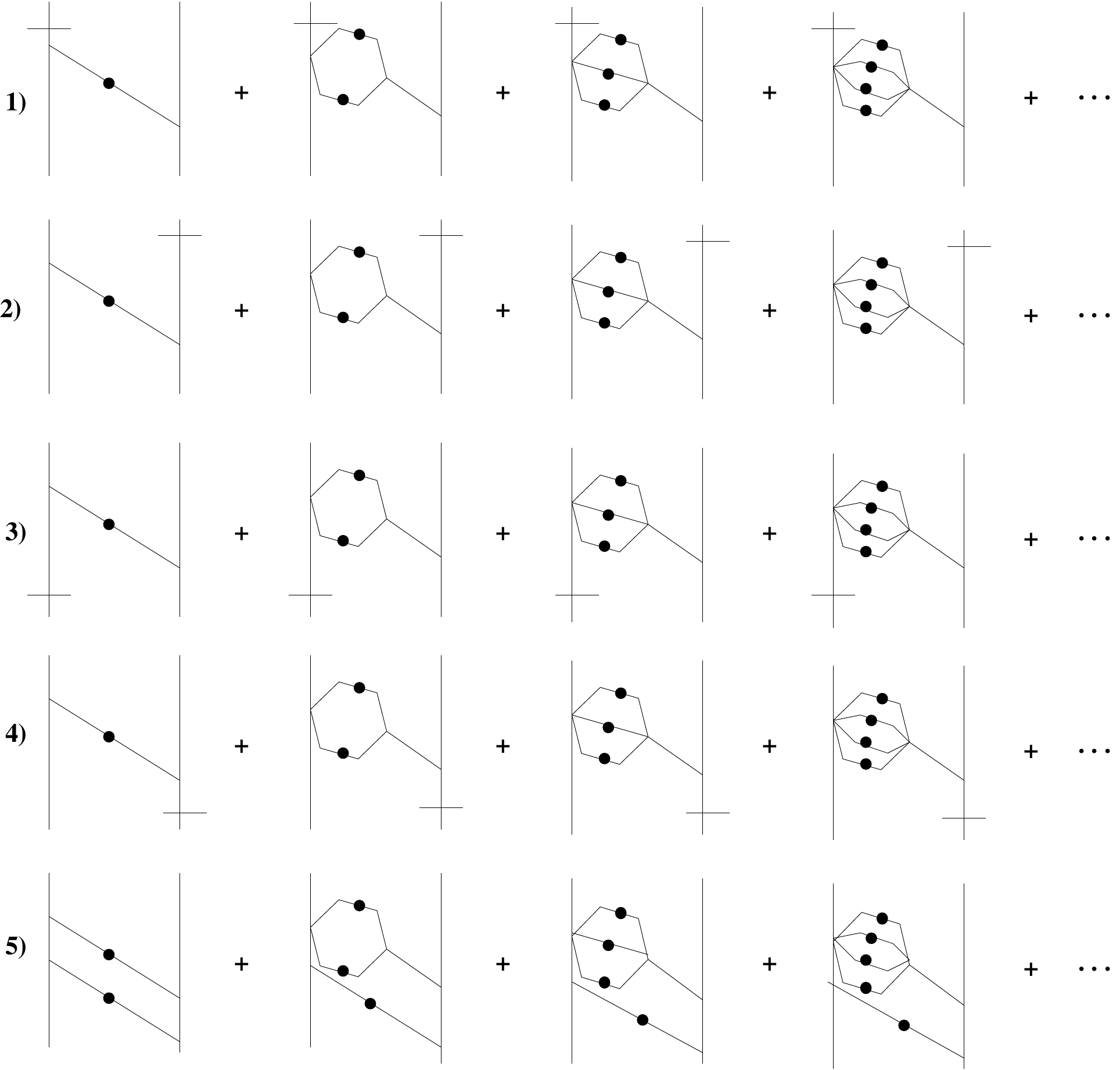}  \\

\caption{Some of therms in terms of $\partial_\tau^2$.  We have left out the terms in the tensor part, as these are known to square to zero.}
\label{figure:twistedc1^2}
\end{figure}

We list some of the terms of $\partial_\tau^2$ in Figure \ref{figure:twistedc1^2}. The fact that $\partial_H$ is a derivation is expressed diagrammatically as in Figure \ref{figure:partialHderivation}.  This relation can be used to add diagrams in the second and fourth rows of Figure \ref{figure:twistedc1^2}.  In place of the coderivation relations, we must use the $C_\infty$ coalgebra relations for $(C_*,\{c_n\})$ .  The relation for $n=3$ is expressed in Figure \ref{figure:partialCrelations}.  We use these relations to add figures in the first and third rows of Figure \ref{figure:twistedc1^2}.  Some of the resulting diagrams will either cancel with diagrams in rows five or higher.  The rest of the diagrams are shown in Figure \ref{figure:MCequations}.  The Maurer Cartan equation is present in each row.   Since $\tau$ is a twisting cochain, the sum to zero.

\begin{figure}
\includegraphics[width=2.5in]{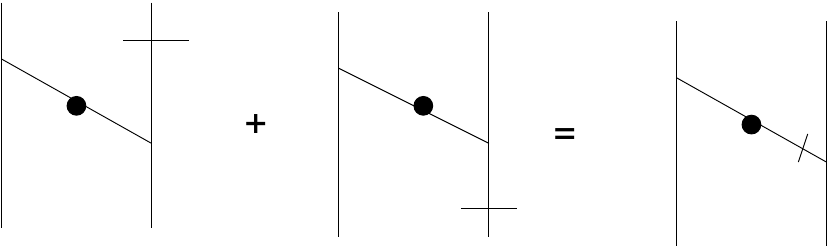}  \\

\caption{The equality here come from the fact that $H_*$ is a differential graded algebra.}
 \label{figure:partialHderivation}
\end{figure}

\begin{figure}
\includegraphics[width=4.5in]{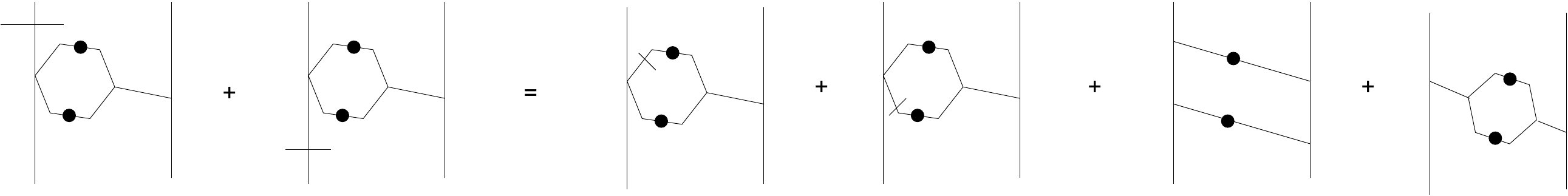}  \\

\caption{The equality here comes from the fact that $(C_*,\{c_n\} )$ is a $C_\infty$ coalgebra.}
\label{figure:partialCrelations}
\end{figure}

\begin{figure}
\includegraphics[width=4in]{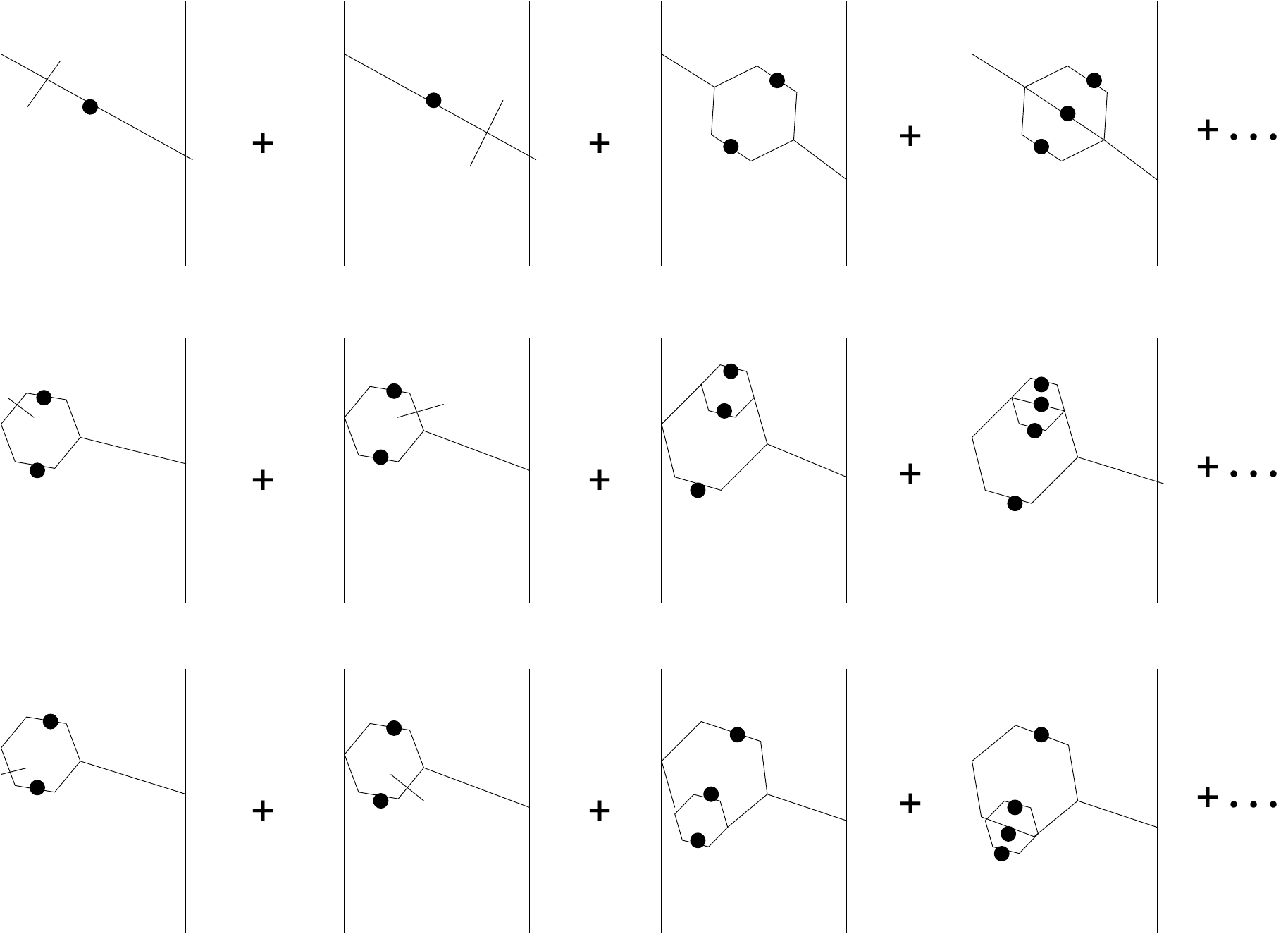}  \\

\caption{These remaining terms in $(\partial_\tau)^2$ sum to zero because $\partial_H \tau + \tau \partial_H + m_2^{Hom}(\tau,\tau)+ m_3^{Hom}(\tau,\tau,\tau) + \cdots =0$.}
\label{figure:MCequations}
\end{figure}

Next, we show that $c_2^\tau$ is a coderivation of $\partial_\tau$.   In Figure \ref{figure:coderivationLHS}, the graphs representing $c_2^\tau \circ \partial_\tau$ are drawn and in Figure \ref{figure:coderivationRHS}  the graphs representing $(\partial_\tau \otimes 1 )\circ c_2^\tau$ are drawn.  The graphs representing $(1 \otimes \partial_\tau) \circ c_2^\tau$ are the same as the graphs representing $(\partial_\tau \otimes 1) \circ c_2^\tau$ except the graphs are connected by the right output edge of each tree as opposed to the left output edge.

Multiplication by a primitive element is a coderivation, which gives
us identities expressed  in Figure
\ref{figure:primitivecoderivation}.  This allows us to compare the
graphs from $c_2^\tau \circ \partial_\tau$ with the graphs from
$(\partial_\tau\otimes 1 + 1 \otimes \partial_\tau)\circ c_2^\tau$.  Note that
on the left hand side of each pairing, we have many compositions of
the form $(1 \otimes \cdots \otimes c_j \otimes \cdots \otimes
1)\circ c_i, $ where $c_i,c_j$ are maps of the $C_\infty$ coalgebra
on $C_*$.  The relations in the $C_\infty$ coalgebra state that
$\sum_{i+j+1= n}(1 \otimes \cdots \otimes c_j \otimes \cdots \otimes 1
)\circ c_i=0$.  Noting which maps in our graphs appear in the sum
and which graphs do not appear, we can apply the $C_\infty$
coalgebra relation to obtain many identities.  When this is done, we
obtain graphs which involve the Maurer Cartan equation for $\tau$, just
as we did in showing $\partial_\tau^2=0$.  Since $\tau$ is a
twisting cochain, this sum is zero and $c_2^\tau$ is a coderivation
of $\partial_\tau$.  In Figure \ref{figure:coalgebrarelations} we
organize the graphs in $c_2^\tau \circ \partial_\tau + (\partial_\tau \otimes
1 + 1 \otimes \partial_\tau)\circ c_2^\tau$.  The relations for the coalgebra structure on
$C_*$ state that the sum of these graphs are equal to the graphs
in Figure \ref{figure:coalgebrarelations2}.  The sum of these graphs is zero, because of the Maurer Cartan equation.

\begin{figure}
\includegraphics[width=4in]{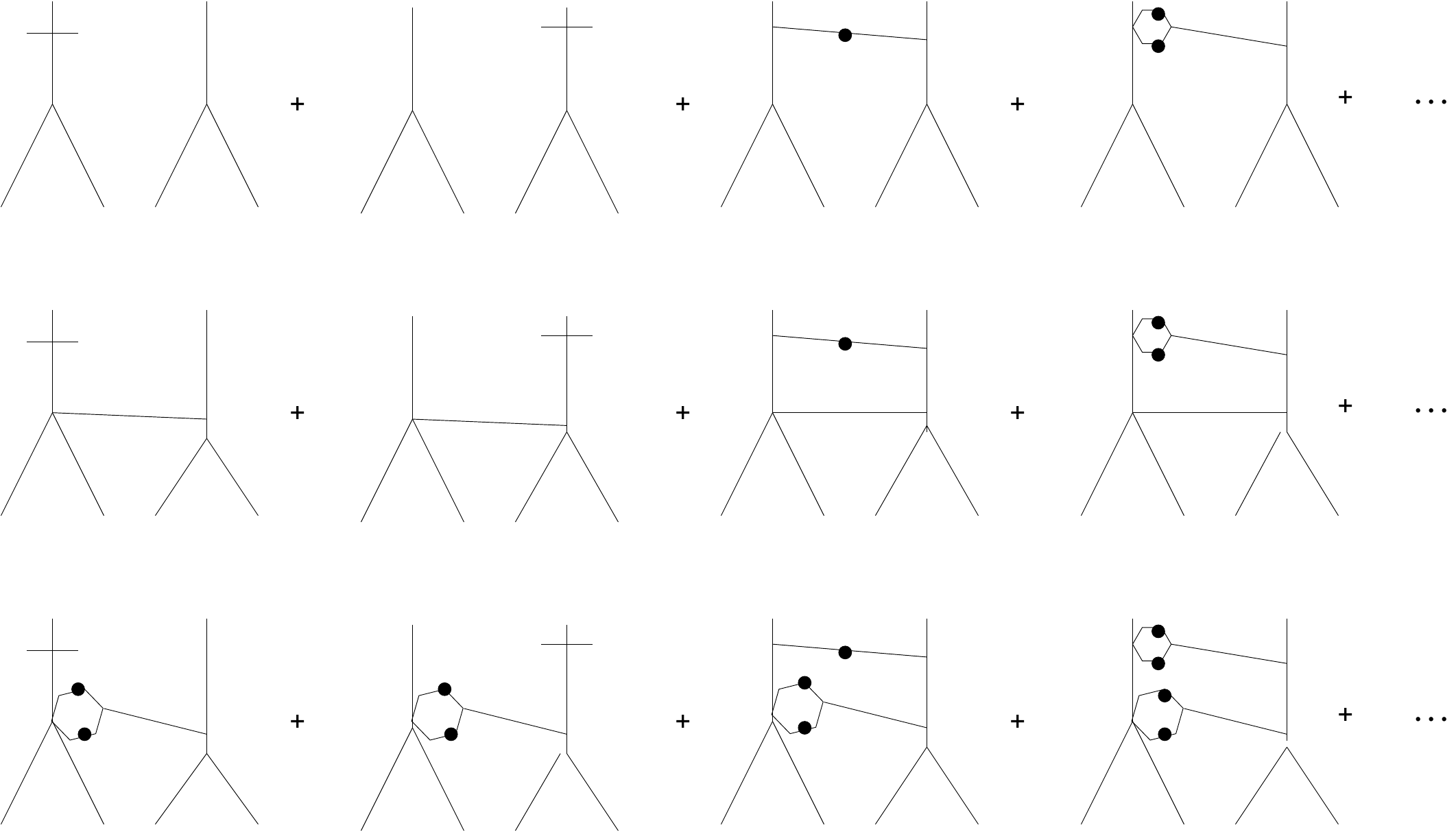}  \\

\caption{The graphs representing $c_2^\tau\circ \partial_\tau.$}
 \label{figure:coderivationLHS}
\end{figure}

\begin{figure}
\includegraphics[width=4in]{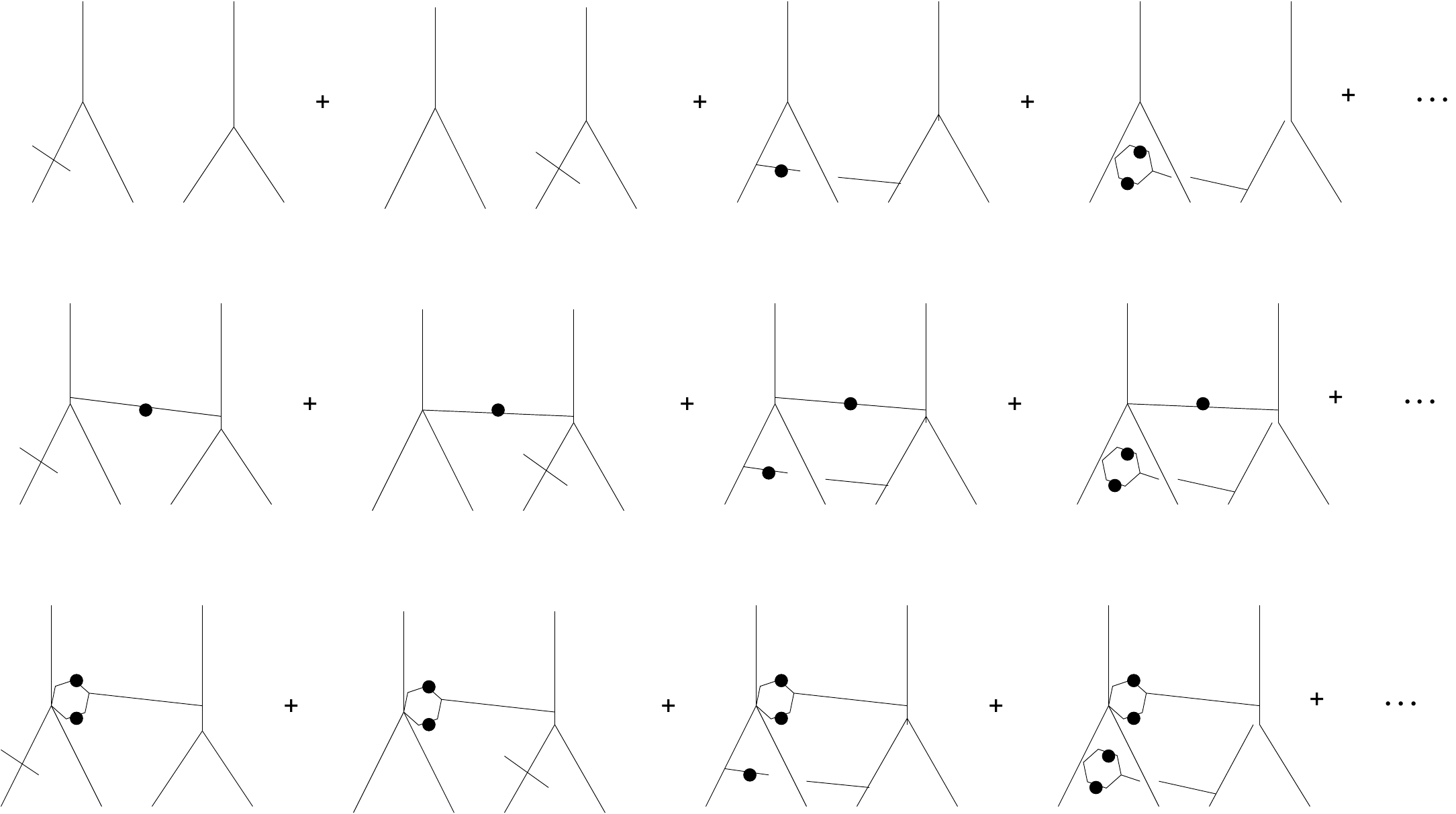}  \\

\caption{The graphs representing $\partial_\tau \otimes 1 \circ c_2^\tau.$}
\label{figure:coderivationRHS}
\end{figure}

The reader can see that this situation generalizes for $n>2$.  In each of these cases, we have many compositions involved in the $C_\infty$ coalgebra relation for $C_*$.  When we replace these graphs, using the coalgebra structure, we obtain graphs involving Maurer Cartan equation.  We summarize the relation in Figure \ref{figure:higherc}.

\begin{figure}
\includegraphics[width=3.5in]{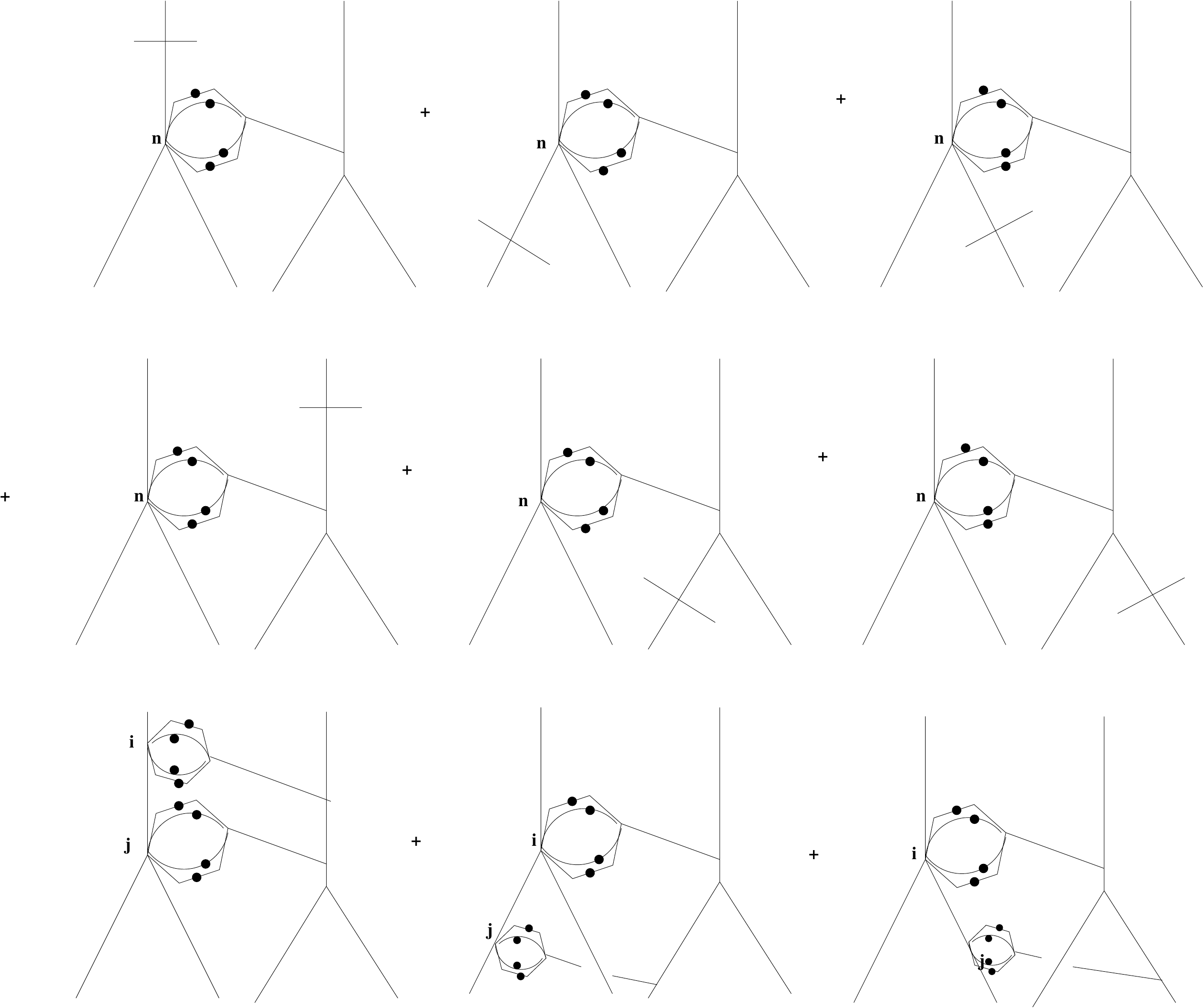}  \\

\caption{The graphs of $c_2^\tau\circ \partial_\tau+(\partial_\tau\otimes1+1\otimes \partial_\tau)c_2^\tau$ organized to show how the $C_\infty$ coalgebra on $C_*$ is used.}
\label{figure:coalgebrarelations}
\end{figure}

\begin{figure}
\includegraphics[width=3in]{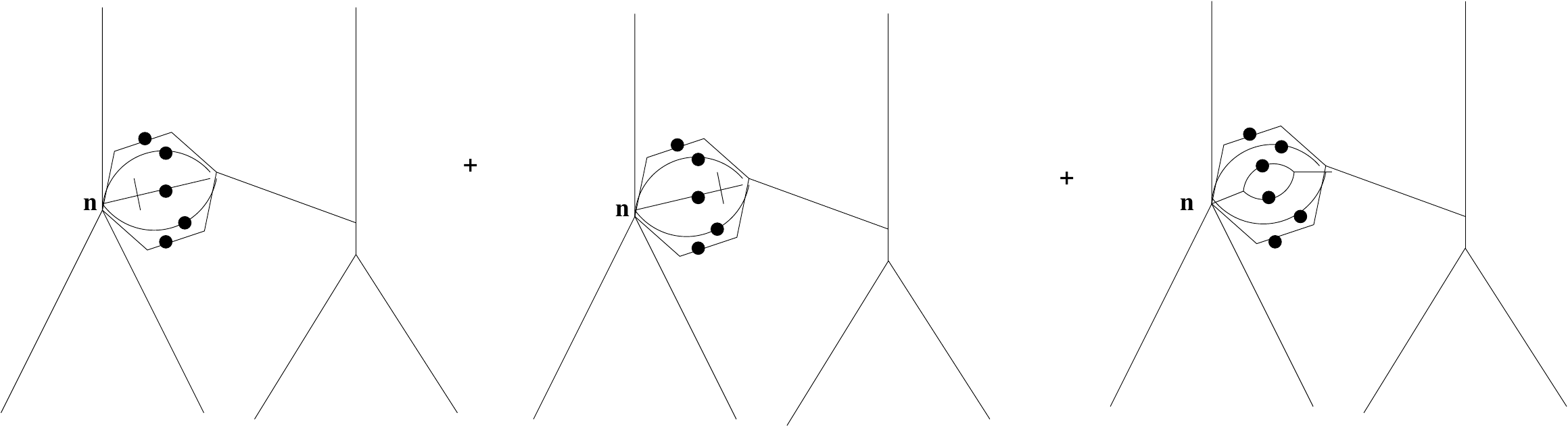}  \\

\caption{These graphs are equal to the graphs in Figure \protect{\ref{figure:coalgebrarelations}} using the $C_\infty$ coalgebra on $C_*$.  Note that these terms involve $\partial_H \tau + \tau \partial_C + \tau \cdot \tau + \tau \cdot \tau \cdot \tau + \cdots =0.$}
\label{figure:coalgebrarelations2}
\end{figure}

\begin{figure}
\includegraphics[width=4.5in]{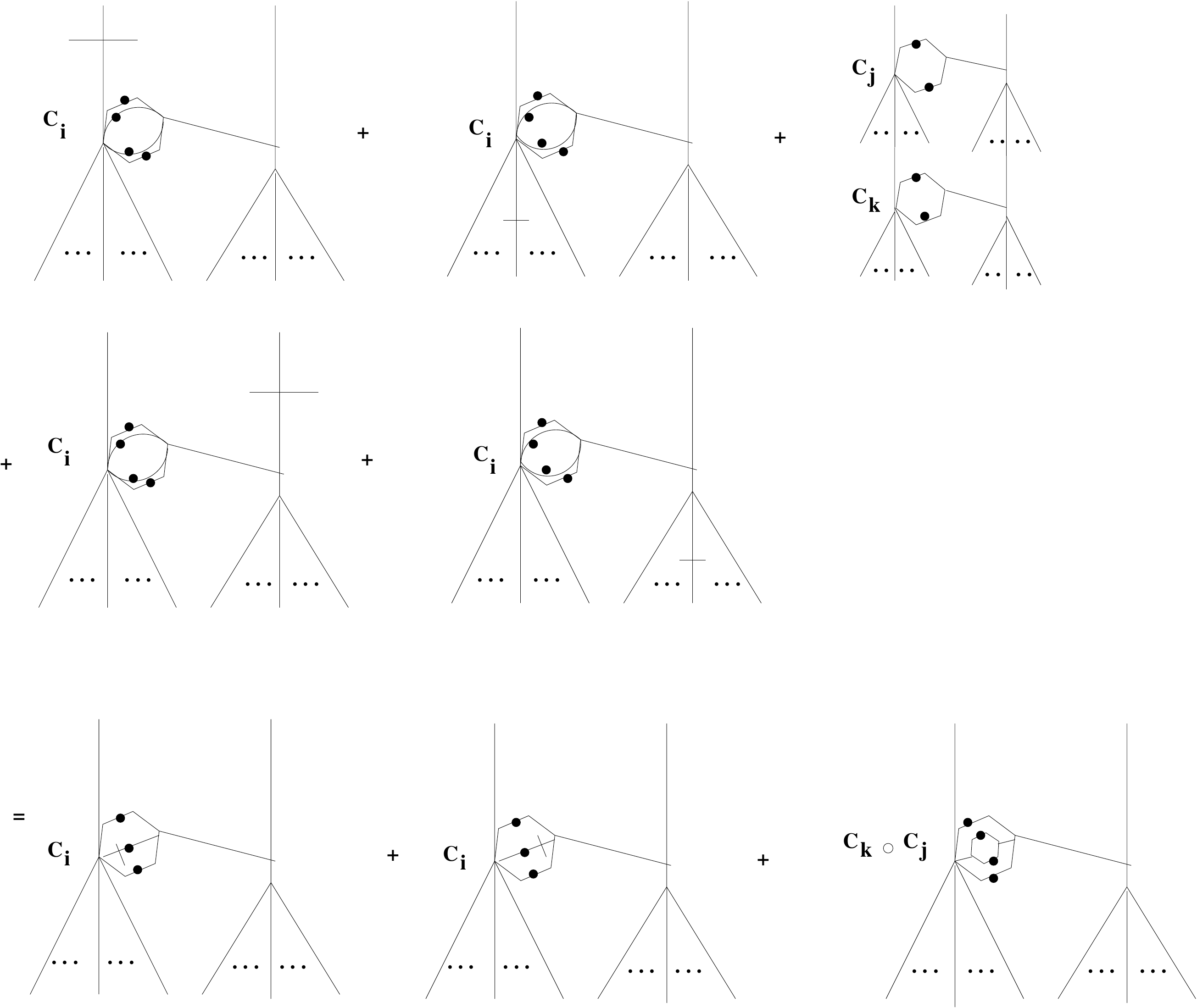}  \\

\caption{To show that $c_n^\tau$ form a coalgebra structure, use the relation above to get a sequence of graphs involving the Maurer Cartan equation.  The equality is due to the fact that $C_*$ is a $C_\infty$ coalgebra.  }
 \label{figure:higherc}
\end{figure}

\end{proof}

\subsection{$  C_\infty\, \text{coalg} \otimes_\tau \, \text{bialgebra}$ as an $A_\infty$ coalgebra using bracket action action}

In the previous section, we used the twisting cochain and left multiplication in $H_*$ to twist the $A_\infty$ coalgebra structure on $C_* \otimes H_*$.  In this section, we consider another action.  For $a \in H_*$, the bracket action of $a$ on $H_*$ is defined by $[a,x]= ax - xa.$  Note that $[a,-]$ is a derivation.  If $a$ is a primitive element, then $[a,-]$ is also a coderivation.  

Given a twisting cochain $\tau:C_* \rightarrow H_*$ such that $Im(\tau) \subset Prim(H_*)$, we define a twisted $A_\infty$ coalgebra structure on $C_* \otimes H_*$.  The process is the same as the one defining the previous twisted $A_\infty$ coalgebra, except we replace the multiplication in $H_*$ with the bracket action.  We use the same notation $\{\partial_\tau, c_2^\tau, c_3^\tau, \cdots \}$ and so we will be explicit when to use left multiplication and when to use the bracket action.

\begin{thm} \label{thm:conjugationcoalgebrabracket}
Let $C_*$ be a $C_\infty$ coalgebra, $H_*$ a dg bialgebra, and $\tau:C_* \rightarrow H_*$ a twisting cochain such that $Im(\tau) \subset Prim(H_*)$.  The maps $\{\partial_\tau, c_2^\tau, c_3^\tau, \cdots \}$, obtained from the bracket action, define an $A_\infty$ coalgebra on $C_*  \otimes H_* $.
\end{thm}

\begin{proof}

The only property of left multiplication used in the proof of Theorem \ref{thm:twistedcoalgebra} is that left multiplication by a primitive element is a coderivation.  Since conjugation by a primitive element is a coderivation, the proof applies to this theorem as well.

\end{proof}

\subsection{$\text{Cyclic} \, C_\infty \, \text{coalg} \otimes_\tau \text{bialg}$ as an $A_\infty$ algebra using bracket action }

Sometimes a $C_\infty$ coalgebra has extra structure on it, allowing
one to define an algebra structure on $C_* \otimes H_*$.  We
consider the case when the coalgebra has a non-degenerate bilinear
form that is compatible with the coalgebra structure, i.e., a cyclic
$C_\infty$ coalgebra.  We review the relevant definitions.

A cyclic $A_\infty$ algebra consists of a finite dimensional $A_\infty$ algebra $(A_*, \{m_n\})$ and a non-degenerate bilinear form $\langle  ,  \rangle: A_* \otimes A_* \rightarrow k$ such that
\begin{eqnarray*}
\langle m_n(x_1, \cdots, x_n), x_0 \rangle = (-1)^N \langle m_n(x_0, \cdots, x_{n-1}), x_n \rangle,
\end{eqnarray*}
where $N=-1+ |x_0|(|x_1| + \cdots + |x_{n}|)$.  The bilinear form defines an isomorphism between $A$ and its dual.  The maps $m_n$ can then be viewed as elements in $A[-1]^{* \otimes n} \otimes A[-1] \cong A[-1]^{\otimes n+1}.$

\begin{lem}\label{lem:cyclicallyinvariant}
Let $(A_*, \{m_n\}, \langle ,  \rangle)$ define a cyclic $A_\infty$ algebra.  Then $m_n \in A[-1]^{\otimes n+1}$ is cyclically invariant.
\end{lem}

\begin{proof}

Let $m_n= \sum x_1 \otimes \cdots \otimes x_{n+1} \in A[-1]^{\otimes n+1}$.  It suffices to show that $x_1 \otimes \cdots \otimes x_{n+1} = x_2 \otimes \cdots \otimes x_{n+1} \otimes x_1.$  This is seen to be the case by expressing $\langle -,-\rangle$ as an element in $A_* \otimes A_*$ and writing the conditions for a cyclic $A_\infty$ algebra in terms of elements in the tensor algebra.

\end{proof}

Viewing the maps $\{m_n\}$ as elements in the tensor and using the Koszul sign rule, one can determine the sign $(-1)^N$ found in the definition of a cyclic $A_\infty$ algebra.  We define a cyclic $A_\infty$ coalgebra viewing $c_n$ as cyclically invariant elements in the tensor product.

\begin{defn}
$(C_*, \{c_n\}, \langle -,- \rangle)$ is a \dfn{cyclic $A_\infty$ coalgebra} if
\begin{enumerate}
\item{$C_*$ is finite dimensional}
\item{$(C_*, \{c_n\} )$ is an $A_\infty$ coalgebra,}
\item{$\langle , \rangle$ is a non-degenerate bilinear form,}
\item{the maps $c_n$  when identified as elements $C_*^{\otimes n+1}$  using the bilinear form, are cyclically invariant. }
\end{enumerate}
\end{defn}

The condition that $C_*$ is finite dimensional implies that  $\langle , \rangle $ defines an isomorphism between $C_*$ and its dual $C^*$.  A cyclic $C_\infty$ coalgebra is defined in the obvious way.  Given a cyclic $C_\infty$ coalgebra $C_*$, the bilinear form and maps $\{c_n\}$ can be used to define a $C_\infty$ algebra $\{m_n:C_*[-1]^{\otimes n} \rightarrow C_*[-1]\}$.  So $C_* \otimes H_*$ has an $A_\infty$ algebra structure given by combining the $C_\infty$ algebra on $C_*$ with the strict algebra structure on $H_*$.  Does the twisting cochain $\tau:C_* \rightarrow H_*$ define a twisted $A_\infty$ algebra on $C_* \otimes H_*$?  We show that it does and unlike in the previous cases, we do not need to twist the higher maps.  In Theorem \ref{thm:dga}, we prove the case when $C_*$ is a strict cyclic coalgebra.  Also, note that since bracketing is always a derivation, whether by a primitive element or not, we do not require $Im(\tau)\subset Prim(H_*)$. If $Im(\tau) \subset Prim(H_*)$ and $H_*$ is a dg Hopf algebra, and not just a dg bialgebra, then the bracket action agrees with another action, which we call the conjugation action.  We use this action in Theorem \ref{thm:loopproduct}.

\begin{thm} \label{thm:loopproductbracket}
Let $C_*$ be a cyclic $C_\infty$ coalgebra,  $H_*$ be a dg bialgebra, and $\tau:C_* \rightarrow H_*$ be a twisting cochain.   The maps $\{\partial_\tau, m_2, m_3, \cdots \}$ defined using the bracket action in $H_*$ give $C_* \otimes_\tau H_*$ the structure of an $A_\infty$ algebra.
\end{thm}

\begin{proof}

Since $\{\partial, m_2, m_3, \cdots \}$ defines an (untwisted) $A_\infty$ algebra, it suffices to show that the twisted terms in $\partial_\tau$ all cancel.  We first show that $\partial_\tau$ is a derivation of $m_2$, \begin{eqnarray} \label{eqnarray:new}
\partial_\tau \circ m_2 = m_2 \circ (\partial_\tau \otimes 1+ 1 \otimes \partial_\tau).  
\end{eqnarray}
We refer the reader to Figures \ref{figure:loopderivationLHS} and \ref{figure:loopderivationRHS} for graphs representing the LHS and RHS of equation (\ref{eqnarray:new}).  Since the bracket action is a derivation, the diagrams in Figure \ref{figure:loopderivationLHS} are equal to the diagrams in Figure \ref{figure:double}.  We need to show that Figure \ref{figure:loopderivationRHS} is equal to Figure \ref{figure:double}.

\begin{figure}
\includegraphics[width=4in]{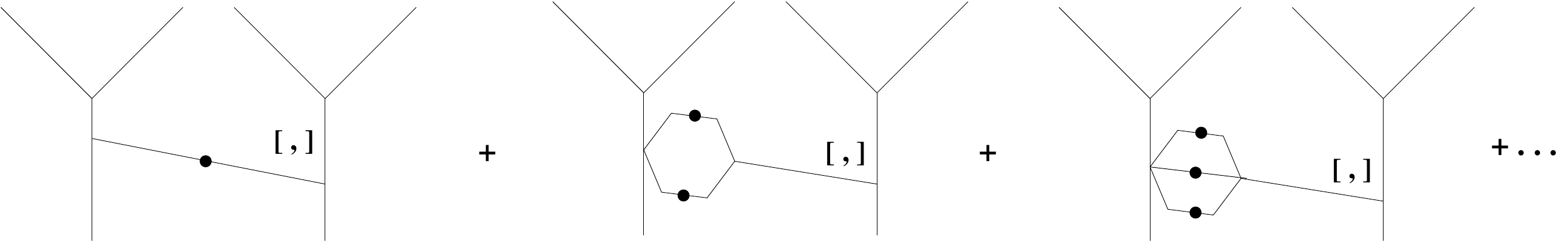}  \\

\caption{A graphical representation of $\partial_\tau \circ m_2$.  The label $[ , ]$ is to remind the reader that the bracket action is applied on $T(H_*(M)[-1])$, and not the product in $T(H_*(M)[-1])$.}
 \label{figure:loopderivationLHS}
\end{figure}

\begin{figure}
\includegraphics[width=4in]{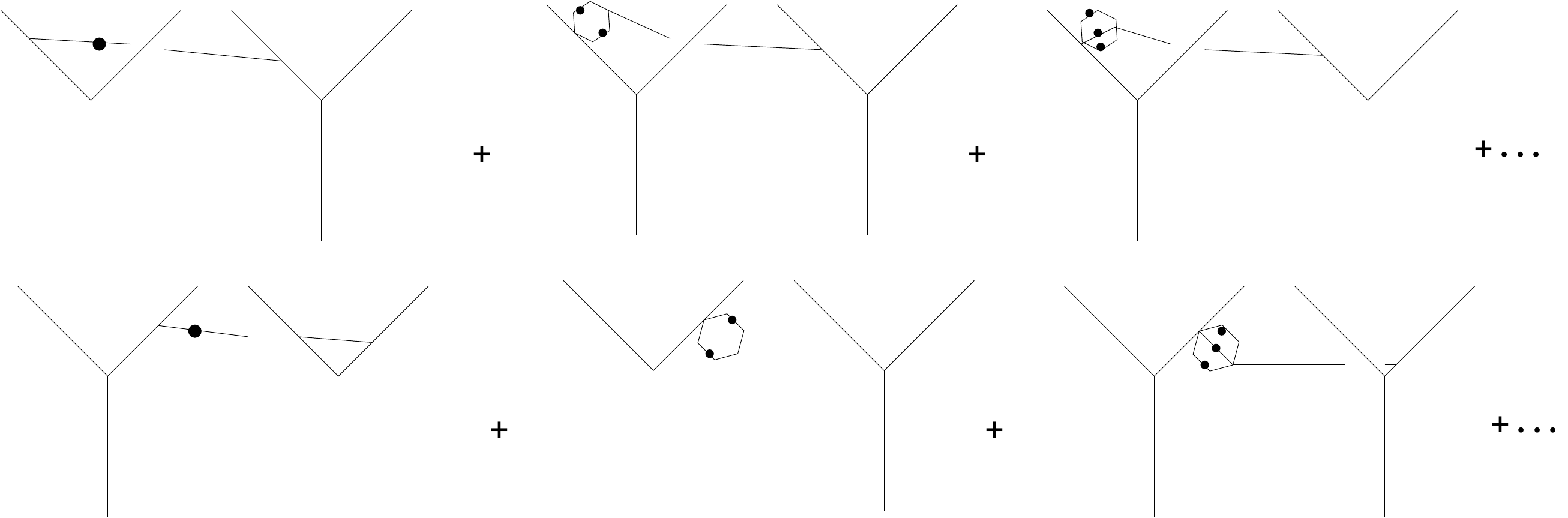}  \\

\caption{A graphical representation of $m_2\circ(\partial_\tau \otimes 1 + 1 \otimes \partial_\tau)$.}
 \label{figure:loopderivationRHS}
\end{figure}

\begin{figure}
\includegraphics[width=4in]{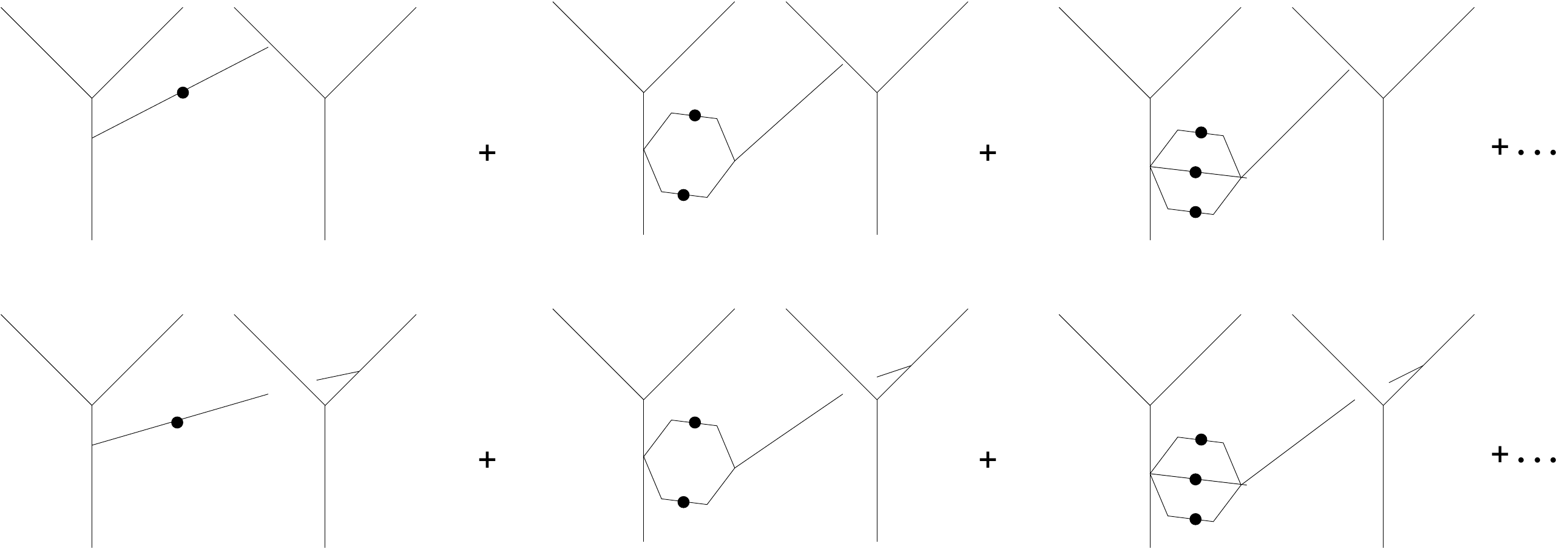}  \\

\caption{Because the bracket action is a derivation, these diagrams are equal to the one found in Figure \ref{figure:loopderivationLHS}.}
 \label{figure:double}
\end{figure}

The LHS of  equation (\ref{eqnarray:new}) has compositions $c_n \circ m_2: C_* [-1]^{\otimes 2} \rightarrow C_* [-1]^{\otimes n}$.  The maps on the RHS has compositions  $ (1^{\otimes i} \otimes m_2 \otimes 1^{\otimes j} )\circ (c_n \otimes 1):C_*[-1]^{\otimes 2} \rightarrow C_*[-1]^{\otimes n}$.  We show these two maps are equal by writing the compositions as elements in $C_*[-1]^{\otimes n+2}$ and using Lemma \ref{lem:cyclicallyinvariant}.

The map $c_n$ can be written as $\sum x_1 \otimes \cdots \otimes x_{n+1} \in C_*[-1]^{\otimes n+1}$ and $m_2$ as an  $\sum y_1 \otimes y_2 \otimes y_3 \in C_*[-1]^{\otimes 3}$.  Their composition $c_n \circ m_2$ is expressed as $$\sum \langle x_1, y_3 \rangle x_2\otimes \cdots x_n\otimes y_1 \otimes y_2 \in C_*[-1]^{\otimes 4}.$$  The composition on the RHS of the equation, $(1^{\otimes i} \otimes m_2 \otimes 1^{\otimes j}) \circ (c_2 \otimes 1)$ is described in the same way except for a different pairing $\langle x_i, y_j \rangle$.  However, since $c_n$ and $m_2$ are cyclically invariant, the compositions are equal.

The higher compatibilities for the $A_\infty$ algebra proceed in exactly the same way, with $m_2$ replaced by $m_l$.

\end{proof}

Given the $A_\infty$ algebra $C_* \otimes H_*$, we can symmetrize the maps to obtain an $L_\infty$ algebra $([C_* \otimes H_*], \{\partial_\tau, l_2, l_3, \cdots \})$.  This restricts to an $L_\infty$ algebra structure on $C_* \otimes Prim(H_*)$.

\begin{thm} \label{thm:linfty}
Let $(C_* \otimes H_*, \{\partial_\tau, m_2, m_3, \cdots \})$ be the $A_\infty$ algebra described in Theorem \ref{thm:loopproductbracket}.  Then $(C_* \otimes Prim(H_*), \{\partial_\tau, l_2, l_3, \cdots \})$, obtained by symmetrizing $\{m_n\}$, is an $L_\infty$ algebra. 
\end{thm}   

\begin{proof}
Since $C_*$ is finite dimensional, we can identify $C_* \otimes H_* \cong Hom(C^*, H_*)$, where $C^*$ is a $C_\infty$ coalgebra.  Then the statement follows from Lemma \ref{lem:Hom(C,L)}.
\end{proof}

This gives an $A_\infty$ algebra structure on  $C_* \otimes H_*$ and an $L_\infty$ algebra structure on $C_* \otimes Prim(H_*)$.  More can be said when $C_*$ is a strict unital commutative algebra.  In this situation, $C_* \otimes Prim(H_*)$ can be viewed as a Lie algebra over $C_*$.  Its universal enveloping algebra over $C_*$, denoted $U_{C_*}(C_* \otimes Prim(H_*))$ is $C_* \otimes H_*.$  Note if we take the universal enveloping algebra of $C_* \otimes Prim(H_*)$ (viewed as a Lie algebra over the ground field), we obtain $U(C_* \otimes H_*)$ which is not equal to $U_{C_*}(C_* \otimes Prim(H_*))$.  We are not aware of the corresponding notion for $C_\infty$ algebras to make the analogous statement.  This seems to be a useful notion.  We discuss the strict case in more detail in Theorem \ref{thm:strictenveloping}.

\subsection{$  C_\infty\, \text{coalg} \otimes_\tau \, \text{Hopf alg}$ as an $A_\infty$ coalgebra using conjugation action} \label{subsection:conjugationaction}

In our applications of the previous results, we would like to relate the twisted algebraic structures to the total space of some bundle.  Let $G \rightarrow P \rightarrow M$ be a principal $G$ bundle and $G \rightarrow Conj(P) \rightarrow M$ be the associated bundle with respect to the conjugation action.  Note that $H_*(M)$ is a cyclic $C_\infty$ coalgebra and $H_*(G)$ a bialgebra, and moreover, a Hopf algebra.  Then given a suitable twisting cochain $\tau:H_*(M) \rightarrow H_*(G)$, we can form the twisted algebraic structures using the methods described above.  The homology of the total space, $H_*(Conj (P))$ can be identified with linear homology of the twisted algebraic structures, that is homology the homology $H_*(M) \otimes H_*(G)$ with respect to $\partial_\tau$.  However, the argument uses Brown's theory of twisting cochains, which requires using the conjugation action.  Because the conjugation action uses the inverse operation in $G$, the algebraic setup in this situation requires $H_*$ to be a dg Hopf algebra.  We will see that when $Im(\tau)\subset Prim(H_*)$, the conjugation action agrees with the bracket action.  

Let $H_*$ be a Hopf algebra.  Denote the antipode map of $H_*$ by $s:H_* \rightarrow H_*$.  Given an element $a \in H_*$, we define the conjugation action of $a$ on $H_*$ by 
\begin{eqnarray*}
conj_a: H_* &\rightarrow & H_* \\
x &\mapsto& \sum a_{(1i)} x s(a_{(2i)}).
\end{eqnarray*}
The homology of a topological group $H_*(G)$ is a Hopf algebra.  The group acts on itself by conjugation, and so induces an action on $H_*(G)$.  The following lemma shows that this action is the same as the conjugation action of the Hopf algebra.

\begin{lem}
Let $G$ be a topological group.  The conjugation action in $G$ induces a map
\begin{eqnarray*}
H_*(G) \otimes H_*(G) &\rightarrow& H_*(G) \\
a \otimes x & \mapsto & \sum a_{(1i)} x s(a_{(2i)}).
\end{eqnarray*}
 \end{lem}

\begin{proof}
Conjugation is described by the composition
\begin{eqnarray*}
&G \times G& \overset{Diag \times 1} {\rightarrow}  G \times G \times G \overset{1 \times inv \times 1}{\rightarrow} G \times G \times G \rightarrow G \times G \times G \rightarrow  G \\
&(x,y)&  \mapsto (x,x,y) \mapsto  (x, x^{-1}, y) \mapsto  (x, y, x^{-1}) \mapsto  xyx^{-1}.
\end{eqnarray*}
The diagonal map in $G$ induces the coproduct $\Delta$ on $H_*(G)$ and the inverse map in $G$ induces the antipode $s$.  This proves the lemma.

\end{proof}

The following lemma shows that conjugation by a primitive element is a coderivation and a derivation.

\begin{lem}\label{lem:conjugation}
Let $H_*$ be a Hopf algebra.
\begin{enumerate}
 \item{Conjugation by a primitive element in a Hopf algebra is a coderivation.}
 \item{Conjugation by a primitive element in a Hopf algebra is a derivation.}
\end{enumerate}
\end{lem}

\begin{proof}
\begin{enumerate}
\item{Let $a$ be a primitive element of $H_*$.  The antipode has to satisfy $ m \circ (1 \otimes s) \circ \Delta(a) =0$, which means $s(a)= -a$.  Then $conj_a(x) = \sum a_{(1i)} x s(a_{(2i)}) = ax - xa$.   This is a coderivation because multiplying by a primitive element is a coderivation.
}
\item{Let $a$ be a primitive element.  Then
\begin{eqnarray*}
conj_a(x)\cdot y + x \cdot conj_a(y) &=& ax y - xay + xay - xya \\
&=& axy - xya \\
&=& conj_a(xy).
\end{eqnarray*}
}
\end{enumerate}
\end{proof}

\begin{thm} \label{thm:conjugationcoalgebra}
Let $C_*$ be a $C_\infty$ coalgebra, $H_*$ a dg Hopf algebra, and $\tau:C_* \rightarrow H_*$ a twisting cochain such that $Im(\tau) \subset Prim(H_*)$.  The maps $\{\partial_\tau, c_2^\tau, c_3^\tau, \cdots \}$, obtained from the conjugation action, define an $A_\infty$ coalgebra on $C_*  \otimes H_* $.
\end{thm}

\begin{proof}
For $a \in Prim(H_*)$, the conjugation action, $conj_a$, and bracket action $[a, ]$ agree.  Since $Im(\tau) \subset Prim(H_*)$, the maps $\{\partial_\tau, c_2^\tau, c_3^\tau, \cdots\}$ defined using the conjugation action are equal to the maps defined using the bracket action.  The statement then follows from Theorem \ref{thm:conjugationcoalgebrabracket} 
\end{proof}

\begin{thm} \label{thm:loopproduct}
Let $C_*$ be a cyclic $C_\infty$ coalgebra,  $H_*$ be a dg Hopf algebra, and $\tau:C_* \rightarrow H_*$ be a twisting cochain with $Im(\tau) \subset Prim (H_*)$.   The maps $\{\partial_\tau, m_2, m_3, \cdots \}$ defined using the conjugation action in $H_*$ give $C_* \otimes_\tau H_*$ the structure of an $A_\infty$ algebra.
\end{thm}

\begin{proof}
Since $Im(\tau) \subset Prim(H_*)$, the twisted $A_\infty$ algebra structure defined using the conjugation action agrees with the twisted $A_\infty$ algebra structure defined using the bracket action..  The proof then follows from Theorem \ref{thm:loopproductbracket}. 
\end{proof}

\subsection{Addendum}

The graphical approach taken above can obscure some sign issues.  In this section, we show that $\partial_\tau$ is a differential without appealing to graphs.  We also look at the strict (non-infinity) versions of the proofs, with the idea that this will also shed some light on the constructions.

We first show that the twisted differential $\partial_\tau$ is indeed a differential, by referencing the work of Chuang and Lazarev.  In \cite{CL}, a twisted $A_\infty$ algebra is also defined, given a Maurer Cartan element.  While their construction is different on the higher maps, it agrees with the twisted differential described in this paper.  

Let $C^*$ be an $A_\infty$ algebra and $A_*$ a strict dg associative algebra.  Then $C^* \otimes A_*$ is an $A_\infty$ algebra, and a twisting cochain is an element $\tau \in C^* \otimes A_*$ satisfying the Maurer Cartan equation $$\partial_C \tau + \partial_H \tau + m_2(\tau,\tau) + m_3(\tau,\tau,\tau)+ \cdots =0. $$  The twisted differential is then 
$$\partial_\tau(x)= \partial_C (x) + \partial_H(x)+ m_2(\tau,x) + m_3(\tau,\tau,x) + \cdots. $$  

This is related to our construction as follows.  Let $C_*$ be an $A_\infty$ coalgebra, $A_*$ a strict dga, and $\tau:C_* \rightarrow A_*$.  We are only looking to define a differential, which is why we do not require a $C_*$ coalgebra and a dg Hopf algebra.  Then the $A_\infty$ algebra $C^*$ used above is the linear dual of the $A_\infty$ coalgebra.  The twisting cochain $\tau:C_* \rightarrow A_*$ can be viewed as an element in $C^* \otimes A_*$ satisfying the Maurer Cartan equation.  The complex $C_* \otimes A_*$ is the $A_*$-dual of $C^* \otimes A_*$.  The two definitions of the twisted differentials can then be related in this way.

\begin{lem}\label{lem:nograph}
Let $C^*$ be an $A_\infty$ algebra, $A_*$ a differential graded algebra, and $\tau\in C^* \otimes A_*$ a twisting cochain.  Then $\partial_\tau^2=0$.
\end{lem}

\begin{proof}
This is a special case of Theorem $2.6$ $(2)a$ in \cite{CL}.  

We  write out some terms in $\partial_\tau^2(x)$.  The elements $\tau, x \in C_* \otimes A_*$ can be written as $\tau = \sum \tau_C \otimes \tau_A$ and $x= \sum x_C \otimes x_A.$  Then 
\begin{eqnarray*}
\partial_\tau(x) &=& (\partial_C x_C) \otimes x_A + (-1)^{|x_C|} x_C \otimes (\partial_A x_A) + m_2(\tau_C , x_C) \otimes \tau_A \cdot x_A  \\
&\,& + m_3(\tau_C, \tau_C, x_C) \otimes \tau_A \cdot \tau_A \cdot x_A + \cdots,
\end{eqnarray*}
  where we dropped the summation for ease of notation.  Applying $\partial_\tau$ a second time yields compositions of the $A_\infty$ algebra maps $\{m_n\}$.  Using the relations for an $A_\infty$ algebra and strict dg algebra, we obtain terms involving the Maurer Cartan equation for $\tau$.  The argument is similar to the one used to prove Theorem \ref{thm:twistedcoalgebra}.  
\end{proof}

Theorem \ref{thm:loopproduct} asserted the existence of a twisted $A_\infty$ algebra on the tensor product.  We review some definitions and then discuss the strict case of the theorem. 

\begin{defn}
A Frobenius algebra structure on $V$ consists of a commutative multiplication and a non-degenerate inner product such that $$\langle a, bc\rangle = \langle ab,c \rangle. $$  Note that a Frobenius algebra is a cyclic $C_\infty$ algebra with $m_n=0$ for $n>2$.
\end{defn}

Using the non-degenerate inner product of a Frobenius algebra, one can turn the multiplication into a comultiplication.  The multiplication and comultiplication satisfy a certain compatibility, which brings us to the notion of what some authors refer to as an open Frobenius algebra \cite{CEG}.

\begin{defn}
An open Frobenius algebra structure on $V$ consists of a commutative multiplication and a cocommutative comultiplication such that the comultiplication is a map of bimodules.  That is, $$\Delta (ab) = \sum a_{(1i)} \otimes a_{(2i)}b = \sum ab_{(1i)} \otimes b_{(2i)}. $$
\end{defn}
Abrams, \cite{A}, proved that unital Frobenius algebras and unital, counital open Frobenius algebras are equivalent.  

\begin{thm}\label{thm:dga}
Let $C_*$ be a dg Frobenius algebra and $H_*$ a dg bialgebra.  Let $\tau:C_* \rightarrow H_*$ be a twisting cochain such that $Im(\tau) \subset Prim(H_*)$.  Then $(C_* \otimes H_*, \partial_\tau) $ is a differential graded algebra.
\end{thm} 

\begin{proof}
To prove the theorem, we need to show that the twisted term is a derivation.  Let $a \otimes b, c \otimes d \in C_* \otimes H_*$, and let $conj_{b}:H_* \rightarrow H_*$ be the conjugation action by $b \in H_*$.  Then we need to show that   
\begin{eqnarray*}
(ac)_{(1i)} \otimes conj_{\tau(ac)_{(2i)}} bd = a_{(1i)} c \otimes (conj_{\tau(a_{(2i)})} b) d + a c_{(1i)} \otimes b (conj_{\tau(c_{(2i)})} d).
\end{eqnarray*}
Since $Im(\tau) \subset Prim(H_*)$ and conjugating by a primitive element is a derivation, the LHS of the equation is 
\begin{eqnarray*}
(ac)_{(1i)} \otimes conj_{\tau(ac)_{(2i)}} bd = (ac)_{(1i)} \otimes (conj_{\tau(ac)_{(2i)}} b) d + ac_{(1i)} \otimes b (conj_{\tau(ac)_{(2i)}} d). 
\end{eqnarray*}
We need to show that $(ac)_{(1i)} \otimes (ac)_{2i} = a_{(1i)} \otimes a_{(2i)}c = ac_{(1i)} \otimes c_{(2i)}. $

Note that this is the condition that the coproduct is a map of bimodules, i.e., an open Frobenius algebra.  If we use the result that Frobenius algebras and open Frobenius algebras are equivalent, we are done.

We use another argument which follows the proof of Theorem \ref{thm:loopproduct}.  Using the non-degenerate inner product, we express the coproduct $\Delta$ as an element in $C_*^{\otimes 3}$.  The multiplication $m_2: C_* \otimes C_* \rightarrow C_*$ is obtained by dualizing the coproduct $C^* \otimes C^* \rightarrow C^*$ and using the isomorphism between $C^*$ and $C_*$.  So $m_2$ is represented by the same element in $C_*^{\otimes 3}.$  Write this element as  $m_2=\Delta = \sum x_{(1i)} \otimes x_{(2i)} \otimes x_{(3i)} \in C_*^{\otimes 3}$.

We need to show that certain compositions of $\Delta$ and $m_2$ are equal.  In writing the compositions of $\Delta$ and $m_2$, we use the subscript  $i$ to represent $m_2$ ($x_{(1i)} \otimes x_{(2i)} \otimes x_{(3i)}$) and the subscript $j$ to represent $\Delta$.  Then compositions are then given by 
\begin{eqnarray*}
\Delta \circ m_2 &=& \sum_{i,j} \langle x_{(3i)}, x_{(1j)} \rangle x_{(1i)} \otimes x_{(2i)} \otimes x_{(2j)} \otimes x_{(3j)} \\ 
(m_2 \otimes 1) \circ (\Delta \otimes 1) &=& \sum_{i,j} \langle x_{(2j)} , x_{(1i)} \rangle x_{(1j)} \otimes x_{(3j)} \otimes x_{(2i)} \otimes x_{(3i)} \\
(m_2 \otimes 1) \circ (1 \otimes \Delta) &=& \sum_{i,j} \langle x_{(3j)} , x_{(2i)} \rangle x_{(1j)} \otimes x_{(2j)} \otimes x_{(1i)} \otimes x_{(3i)}.
\end{eqnarray*}
Since $m_2= \Delta$ are cyclically invariant, we get the necessary equalities.

\end{proof}

In our construction of a twisted $A_\infty$ algebra structure on $C_* \otimes H_*$, we used a cyclic $C_\infty$ coalgebra. A cyclic $C_\infty$ algebra is the homotopy version of a Frobenius algebra.  It should be possible to define a twisted $A_\infty$ algebra using the homotopy version of an open Frobenius algebra.  The Koszul Duality theory for dioperads, described in \cite{G}, and for properads, described in \cite{V}, provides a definition for such an object.  The dioperad  describing Lie bialgebras, denoted $BiLie$, and the dioperad describing open Frobenius algebras, denoted $BiLie^!$, are Koszul dual (\cite{G} Corollary $5.10$).  So a resolution for $BiLie^!$ is obtained by taking the cobar dual of $BiLie$, denoted $D(BiLie)$, and an open $Frob_\infty$ algebra structure on $V$ is a map of differential graded dioperads $D(BiLie) \rightarrow End(V),$ where $End(V)$ is the endomorphism dioperad.  

The cohomology of a Poincare Duality space is a cyclic $C_\infty$ algebra.  An open manifold is not a Poincare Duality space, but its cohomology is an open Frobenius algebra.  The constructions using cyclic $C_\infty$ algebra would define string topology operations for Poincare Duality spaces, and the constructions using open $Frob_\infty$ algebras would define string topology operations for open manfiolds.   

Theorem \ref{thm:linfty} stated that the $L_\infty$ algebra structure on $C_* \otimes H_*$ restricts to $C_* \otimes Prim(H_*).$  In the strict case, more can be said about the relation between the associative algebra $C_* \otimes H_*$ and the Lie algebra $C_* \otimes Prim(H_*)$.  Let $U_{C_*}(C_* \otimes Prim(H_*))$ be the universal enveloping algebra of $C_* \otimes Prim(H_*)$ viewed as a Lie algebra over $C_*$.  Recall, if $A_*$ is an associative algebra, then $[A_*]$ is the Lie algebra obtained by symmetrizing the mulitplication.  

\begin{thm}\label{thm:strictenveloping}
The Lie bracket on $[C_* \otimes H_*]$ restricts to $C_* \otimes Prim(H_*)$.  Moreover, if $C_*$ is unital, $U_{C_*}(C_* \otimes Prim(H_*)) =C_* \otimes H_*$. 
\end{thm}

\begin{proof}
We first show that the Lie bracket on $[C_* \otimes H_*]$ fixes $C_* \otimes Prim(H_*)$.  This is a simple computation
\begin{eqnarray*}
[a_1 \otimes b_1, a_2 \otimes b_2] &=& a_1 a_2 \otimes b_1 b_2 - a_2 a_1 \otimes b_2 b_1 \\
&=& a_1 a_2 \otimes (b_1 b_2 - b_2 b_1) \\
&=& a_1 a_2 \otimes [b_1, b_2], 
\end{eqnarray*}
where the bracket is in $[H_*]$.  Since $Prim(H_*)$ is a Lie subalgebra of $[H_*]$, this proves the claim.

For the second part, suppose $C_*$ is unital.  Then an element in $U_{C_*}(C_* \otimes Prim(H_*))$ can be re-written
$$(c_1 \otimes h_1) \otimes_{C_*} \cdots \otimes_{C_*} (c_n \otimes h_n) =( c_1 \cdots c_n \otimes h_1) \otimes_{C_*} (1 \otimes h_2) \otimes_{C_*} \cdots \otimes_{C_*} (1 \otimes h_n).  $$  The claim then follows from the construction of the universal enveloping algebra as a quotient of the tensor algebra.

\end{proof}

\section{Application to Spaces}\label{section:stringtopologyoperations}

To describe string topology operations, we start with the path space fibration $\Omega_b(M) \rightarrow P_b(M) \rightarrow M$.  The based loop space $\Omega_b(M)$ is homotopy equivalent to a topological group, so we view $\Omega_b(X)$ as a topological group and the path space fibration as a principal $\Omega_b(M)$ bundle.  The group acts on itself by conjugation and the associated bundle with respect to this bundle, which we refer to as the conjugate bundle, is a model for the free loop space.

\begin{lem}
The conjugate bundle $\Omega_b(M) \rightarrow Conj (P_b(M)) \rightarrow M$ is equivalent to the free loop space bundle $\Omega_b(M) \rightarrow LM \rightarrow M.$
\end{lem}

\begin{proof}

The total space $Conj (P_b (M))$ is $P_b(M) \times_{\Omega_b (M)} \Omega_b(M).$  We define a bundle map from $Conj( P_b(M)) \rightarrow LM$.  Let $[p, a]$ be an element in $P_b(M) \times_{\Omega_b (M)} \Omega_b(M)$ and choose a representative $(p,a)$, where $p:[0,1] \rightarrow M$ and $a:S^1 \rightarrow M$.  Then consider the map $f:[p,a] \mapsto   p a  p^{-1}.$  This map is well defined since  a different representative will be of the form $(p  g, g^{-1} a g)$, which gets sent to
$$( p g ) (g^{-1} a g) (p  g)^{-1} = p a p^{-1}. $$

If $f$ maps fibers isomorphically onto fibers, then $f$ will be a homeomorphism (see for example \cite{MS}, Lemma $2.3$).  Let $F_x(Conj)$ be the fiber of $Conj P_b(M)$ above the point $x \in M$.  An element in the fiber is of the form $[p,a]$ where $p$ is a path from $b$ to $x$ and $a$ is a loop at $b$.  Let $\alpha \in F_x(LM)$ be an element in the fiber of the free loop space bundle.  Then letting $p$ be any path from $b$ to $x$ and $a= p^{-1} \alpha p$, then $f[p, p ^{-1}\alpha p] = \alpha$.

\end{proof}

\subsection{Power Series Connection}\label{subsection:psc}

To apply the theorems proved in Section \ref{section:algebraicsetting}, we need to construct a twisting cochain.  There are several different constructions available for this purpose.  The commutative algebra structure on $\Omega^*(M)$ defines a $C_\infty$ algebra on $H^*(M)$, (see \cite{CG} for a description of how to transfer structure).  The $C_\infty$ algebra defines a derivation of square zero on $\mathcal{L}(H_*(M)[-1])$ and the inclusion $H_*(M) \hookrightarrow  \mathcal{L}(H_*(M)[-1])$ defines a twisting cochain.  Note that the $C_\infty$ algebra on $H^*(M)$ is a minimal model for $\Omega^*(M)$.  Kadeishvilli's Minimal Model Theorem, \cite{K}, provides another construction of a twisting cochain.        

We choose to review the work of Chen \cite{C} and Hain  \cite{H2} on power series connections, which gives an equivalent construction of the minimal model for $\Omega^*(M)$ as the one described above.  A power series connection will be a twisting cochain from $H_*(M) \rightarrow \mathcal{L}(H_*(M)[-1])$ in slightly different terminology.  The equivalence of Kadeishvilli's construction and Hain's construction is described in \cite{Hue}.  The construction is explicit and self contained, which is why we have chosen to include it.

Let $M$ be a simply connected manifold.  We introduce some notation.  If $L$ is a Lie algebra, let $I^2L= [L,L]$, and for $s>2$, $I^sL=[L,I^{s-1}L]$.  Also, for $w \in \Omega^*(M)$, let $J (w) =(-1)^{|w|}w. $

Hain, \cite{H2}, defines a \emph{power series connection} to be a pair  consisting of an element $\omega \in \Omega^*(M) \otimes \mathcal{L}(H_*(M;\R)[-1])$ and derivation $\partial$ on $\mathcal{L}(M_*(X;\R)[-1] )$ , such that
\begin{enumerate}
\item{$\partial^2=0$}
\item{if $\omega \equiv \sum W_i X_i (\text{mod}\, \Omega^*(M) \otimes I^2\mathcal{L}(H_*(M)[-1]))$, then $W_i$ are closed forms whose cohomology classes form a basis for $H^*(M;\R)$, }
\item{$\partial \omega + d \omega - \frac{1}{2}[J\omega,\omega]=0.$ }
\end{enumerate}

The last condition for $\omega$ is referred to as the \emph{twisting cochain condition}.

We go through Hain's construction of a power series connection, which requires the next lemma.  The statement can be found in \cite{H1}, where a dual statement is proved.

\begin{lem}[\cite{H1}, Lemma $3.8$] \label{lem:hainslemma}
Let $L$ be a graded Lie algebra and $\partial$ be a derivation of $L$ such that $\partial(L) \subset [L,L]$.  Suppose $\omega$ is an element of $\Omega^*(M) \otimes L$  such that
\begin{enumerate}
\item{$\omega \equiv \sum W_i X_i (\text{mod}\, \Omega^*(M) \otimes I^2L)$, where $W_i$ are closed forms whose cohomology classes form a linear basis for $H^*(M)$,}
\item{$
\partial \omega + d \omega -\frac{1}{2}[J \omega, \omega] \equiv 0 (\text{mod} \, \Omega^*(M) \otimes I^nL)$,}
\end{enumerate}
Then
\begin{enumerate}
\item{$\partial^2 \equiv 0 (\text{mod} \, I^{n+1} L)$ and}
\item{$d(\partial \omega + d \omega - \frac{1}{2} [J\omega, \omega]) \equiv 0 (\text{mod}\, \Omega^*(M) \otimes I^{n+1} L ). $}
\end{enumerate}
\end{lem}

 \begin{thm}[ \cite{H2}. Theorem $2.6$ ] \label{thm:powerseriesconnection}
There exists a pair $(\omega, \partial)$ such that
\begin{enumerate}
\item{$\omega \in \Omega^*(M) \otimes \mathcal{L}(H_*(M)[-1])$,}
\item{$\partial$ is a derivation of $\mathcal{L}(H_*(M)[-1])$ of square zero}
\item{$\partial \omega + d\omega - \frac{1}{2}[J\omega, \omega]=0.$}
\end{enumerate}
\end{thm}

\begin{proof}

The proof can be found in \cite{H2}.  But we go over it, because this construction will be referred to later on.  Let $(X_i)$ be a basis of $H_*(M)$.  Suppose $(W_i)$ are closed forms in $\Omega^*(M)$ whose cohomology classes form a basis of $H^*(M)$ dual to $(X_i)$.  We construct $\partial$ and $\omega$ inductively and simultaneously.  For ease of notation, let $L= \mathcal{L}(H_*(M)[-1])$.

The first step is to let
\begin{eqnarray*}
\omega_1 &=& \sum_i W_i X_i \\
\partial_1 X_i &=& 0 \, \text{for all} \, i.
\end{eqnarray*}
Then the Maurer Cartan equation is partially satisfied,
$$\partial_1 \omega_1 + d \omega_1 - \frac{1}{2}[J\omega_1, \omega_1] \equiv 0 \,(\text{mod}\, \Omega^*(M) \otimes I^2L). $$
Now, suppose that $\partial_r$ and $\omega_r$ for $r<s$ are defined so that
\begin{enumerate}
\item{$\partial_r$ is a derivation of $L$,}
\item{$\partial_{s-1} X_i \equiv \partial_rX_i \, (\text{mod} \, I^{r+1}L) $}
\item{$\omega_{s-1} \equiv \omega_r \, (\text{mod} \, \Omega^*(M) \otimes I^{r+1}L),$}
\item{$\partial_r \omega_r + d \omega_r - \frac{1}{2}[J\omega_r, \omega_r] \equiv 0 \, (\text{mod}\, \Omega^*(M) \otimes I^{r+1}L). $}
\end{enumerate}

We need to define $\partial_s$ and $\omega_s$ to continue the induction step.  By Lemma \ref{lem:hainslemma}, $$d\left(\partial_{s-1} \omega_{s-1} + d\omega_{s-1} - \frac{1}{2}[J\omega_{s-1}, \omega_{s-1}]\right)=0. $$  But since the cohomology classes of $(W_i)$ form a basis, we have the identity
\begin{eqnarray*}
&\,& \partial_{s-1} \omega_{s-1} + d\omega_{s-1} - \frac{1}{2} [J \omega_{s-1}, \omega_{s-1}]\\
&=& \sum_{i_1 \cdots i_s} \left( \sum_i a_i^{i_1 \cdots i_s} W_i + dW_{i_1 \cdots i_s}   \right) [X_{i_1}, [X_{i_2}, \cdots [X_{i_{s-1}}, X_{i_s}]] ].
\end{eqnarray*}

Then let
\begin{eqnarray*}
\omega_s&=& \omega_{s-1} + \sum_{i_1 \cdots i_s} W_{i_1 \cdots i_s} [X_{i_1}, [X_{i_2}, \cdots ,[X_{i_{s-1}}, X_{i_s}]]] \\
\partial_s X_i&=& \partial_{s-1} X_i + \sum_{i_1 \cdots i_s} a_i^{i_1 \cdots i_s} [X_{i_1}, [X_{i_2}, \cdots [X_{i_{s-1}}, X_{i_{s}}]]].
\end{eqnarray*}
Looking at the Maurer Cartan equation modulo $\Omega^*(M) \otimes I^{s+1}L$,
\begin{eqnarray*}
&\,& \partial_s \omega_s + d \omega_s - \frac{1}{2}[J\omega_s, \omega_s] \\
&\equiv& \partial_{s-1} \omega_{s-1} + d \omega_{s-1} - \frac{1}{2}[J\omega_{s-1},\omega_{s-1}] \\
 &\,& + \sum_i \left( \sum_{i_1 \cdots i_s} a_i^{i_1 \cdots i_s} W_i + d W_{i_1 \cdots i_s} \right) [X_{i_1}, [X_{i_2}, \cdots [X_{i_{s-1}}, X_{i_s}]]]\\
&\equiv&0.
\end{eqnarray*}
This allows us to continue our induction.  Define $\omega$ and $\partial$ by the equations
\begin{eqnarray*}
\partial X_i &\equiv & \partial_s \, (\text{mod} \, I^{s+1}L) \\
\omega &\equiv & \omega_s \, (\text{mod}\, \Omega^*(M) \otimes I^{s+1}L ).
\end{eqnarray*}
\end{proof}

It is a result of rational homotopy theory that the homology of $(\mathcal{L}(H_*(M)[-1]), \partial)$ is isomorphic to $\pi_*(M) \otimes \Q$ and the homology of $(U(\mathcal{L}(H_*(M)[-1])), \partial)$ is isomorphic to $H_*(\Omega_b(M))$ as a Hopf algebra.

The twisting cochain will be the inclusion $H_*(M) \hookrightarrow \mathcal{L}(H_*(M)[-1]).$  The power series connection defines the differential on $\mathcal{L}(H_*(M)[-1])$ to be used in the Maurer Cartan equation and the twisting cochain condition implies that the inclusion is indeed a twisting cochain.  The power series connection also has the following consequence.

\begin{thm}\label{thm:psctocoalgebramap}\cite{GLS}
The power series connection $\omega$ defines a dg coalgebra map $T(H^*(M)[1]) \rightarrow T(\Omega^*(M)[1]).$  There is map $T(\Omega^*(M) [1]) \rightarrow T(H^*(M) [1])$ such that the composition of the two maps is homotopic to the identity on $T(\Omega^*(M)[1])$ and equal to the identity on $T(H^*(M)[1])$.
\end{thm}

\begin{proof}
The element $\omega$ defines a map $T(H^*(M)[1]) \rightarrow \Omega^*(M)$, using the adjunction between tensor and $Hom$.  The twisting cochain condition on $\omega$ implies that the map satisfies the Maurer Cartan equation.  The relations between power series connections and twisting cochains is described in [\cite{GLS}, Section $1.3$].  Using the correspondence between twisting cochains and coalgebra maps then implies that extending the map as a coalgebra respects the differentials.

The second claim about the map $T(\Omega^*(M)[1]) \rightarrow T(H^*(M)[1])$ is a consequence of the map being a deformation retraction.  This result can be found in \cite{Mer}.  

\end{proof}

\subsection{$A_\infty$ coalgebra modeling the homology of the principal path space}

With a twisting cochain $H_*(M) \rightarrow \mathcal{L}(H_*(M)[-1])$ at our disposal, we can apply the theorems of Section \ref{section:algebraicsetting} to the path space fibration and its conjugate bundle.  This gives us three structures, a twisted $A_\infty$ coalgebra on $H_*(M) \otimes T(H_*(M)[-1])$ modeling the coproduct on $H_*(P_b(M))$, a twisted $A_\infty$ coalgebra on $H_*(M) \otimes T(H_*(M)[-1])$ with the conjugation action modeling $H_*(LM)$ modeling the coproduct on $H_*(LM)$, and a twisted $A_\infty$ algebra on $H_*(M) \otimes T(H_*(M)[-1])$ modeling the loop product.

\begin{thm}
Let $M$ be a simply connected manifold, $\Omega_b(M) \rightarrow P_b(M) \rightarrow M$ be the path space fibration, and $H_*(M) \hookrightarrow \mathcal{L}(H_*(M)[-1])$ be the twisting cochain given by the inclusion.  Then $(H_*(M) \otimes T(H_*(M)[-1]), \{c_n^\tau\})$ defines an $A_\infty$ coalgebra model $H_*(P)$.
\end{thm}

\begin{proof}

The diagonal map $M \rightarrow M \times M$ defines a $C_\infty$ coalgebra on $H_*(M)$ and $T(H_*(M)[-1])$ is a Hopf algebra model for $H_*(\Omega_b(M)).$  The theorem is then a consequence of Theorem \ref{thm:twistedcoalgebra}.

\end{proof}

\subsection{$A_\infty$ coalgebra modeling the homology of the free loop space}

This brings us to defining operations in string topology.  The tensor product $H_*(M) \otimes T(H_*(M)[-1])$ is an $A_\infty$ coalgebra given by combining the $C_\infty$ coalgebra on $H_*(M)$ and the strict associative algebra on $T(H_*(M)[-1])$.  Using our twisting cochain, we twist the $A_\infty$ coalgebra as described in Section \ref{subsection:conjugationaction}.

\begin{thm}
Let $H_*(M)$ be a simply connected manifold.  Consider the $C_\infty$ coalgebra on $H_*(M)$, the Hopf algebra on $T(H_*(M)[-1])$, and the conjugation action on $T(H_*(M)[-1])$.  The maps 
\begin{eqnarray*}
\partial_\tau:H_*(M) \otimes T(H_*(M)[-1]) &\rightarrow & H_*(M) \otimes T(H_*(M)[-1]) \\
c_n^\tau:H_*(M) \otimes T(H_*(M)[-1]) &\rightarrow &(H_*(M) \otimes T(H_*(M)[-1])^{\otimes n}
\end{eqnarray*}
define an $A_\infty$ coalgebra.  The linear homology, $(H_*(M) \otimes T(H_*(M)[-1]), \partial_{\tau})$, is the homology of the free loop space of the manifold $H_*(LM)$.
\end{thm}

\begin{proof}

The proof follows from the application of Theorem \ref{thm:twistedcoalgebra}.

\end{proof}

\subsection{$A_\infty$ algebra modeling the homology of the free loop space}

The loop product in $H_*(LM)$, first described in \cite{CS}, is intuitively defined as combining the intersection product of $H_*(M)$ with loop concatenation in $H_*(\Omega_b(M))$.  The set-up of twisted tensor products accommodates such a description.  The tensor product $H_*(M) \otimes T(H_*(M)[-1])$ is an $A_\infty$ algebra.  The map $$m_2:(H_*(M) \otimes T(H_*(M)[-1]))^{\otimes 2} \rightarrow H_*(M) \otimes T(H_*(M)[-1])$$ is a combination of the intersection product and loop concatenation.  However, its linear homology is not $H_*(LM)$ so it does not define an operation in $H_*(LM)$.  For this we need to take the twisted differential $\partial_\tau$.  Unlike the coalgebra case, we do not need to twist the higher multiplication maps.

\begin{thm}
Let $M$ be a simply connected manifold.  Consider the cyclic $C_\infty$ coalgebra on $H_*(M)$, the Hopf algebra on $T(H_*(M)[-1])$, and the conjugation action on $T(H_*(M)[-1])$. The maps
\begin{eqnarray*}
\partial_\tau:H_*(M) \otimes T(H_*(M)[-1])& \rightarrow& H_*(M) \otimes T(H_*(M)[-1])\\
m_n:(H_*(M) \otimes T(H_*(M)[-1]))^{\otimes n} &\rightarrow& H_*(M) \otimes T(H_*(M)[-1])
\end{eqnarray*}
define an $A_\infty$ algebra on $H_*(M) \otimes T(H_*(M)[-1])$. 
\end{thm}

\begin{proof}
The proof is an application of Theorem \ref{thm:loopproduct}.
\end{proof}

\begin{ex}\label{ex:liegroup}
Let $M=G$ be a connected Lie group and consider the path space fibration, $\Omega_b(G) \rightarrow P_b(G) \rightarrow G$.  We claim that the conjugation action of $\Omega_b(G)$ is trivial, and so there is no twisting given by the twisting cochain $H_*(G) \hookrightarrow \mathcal{L}(H_*(G)[-1])$.  Consequently, the string topology operations are given by the untwisted tensor $H_*(G) \otimes T(H_*(G)[-1])$.

To see that the conjugation action is trivial, recall that a Hopf algebra $H_*$
is commutative if the Lie bracket on $Prim(H_*)$ is
zero.  In this case, the Hopf algebra is $H_*(\Omega_b(G))$.  There is a homotopy equivalence, $\Omega_b(G) \cong \Omega_b^2(BG)$.  The Lie
bracket is the same as the Samelson bracket on
$\pi_*(\Omega_b^2(BG))$ which is equal to the Whitehead bracket on
$\pi_*(\Omega_b(BG))$.  This bracket is zero because the Whitehead
bracket is trivial on $H$-spaces.  Since the multiplication is
commutative, the conjugation action is trivial and there is no
twisting coming from a twisting cochain.  This computation agrees
with that in \cite{Hep}.  In that paper, Hepworth uses the
isomorphism between $LG $ and $ G \times \Omega_b(G)$ to determine the Batalin
Vilkovisky algebra on $H_*(\Omega_b(G))$.  Menichi, in
\cite{M}, investigates the BV structure on $H_*(\Omega_b^2(BG))
\otimes H_*(M)$, and also considers the case when $M=G$.  In that
paper, he constructs a BV algebra morphism $H_*(\Omega_b(G))
\rightarrow H_*(\Omega_b(G) \otimes H_*(M) \rightarrow H_*(LM).$

The argument that the conjugation action is trivial can be applied to any manifold $M$ that is an $H$-space.

\end{ex}

\section{Application to Principal $G$ Bundles}

We are interested in applying the results in Section \ref{section:algebraicsetting} to the case of a principal $G$ bundle $G \rightarrow P \rightarrow M$.  This will turn out to be representations of the algebraic structures on $H_*(M) \otimes_\tau T(H_*(M)[-1])$ given in the previous section.  Given a connection on a bundle $G \rightarrow P \rightarrow M$, we get a map of bundles $P_b(M) \rightarrow M$ to $P \rightarrow M$ in the following way.  Choose a basepoint above the fiber in $P \rightarrow M$, and denote it by $e \in F_b(M)$.  Then the fiber can be identified with $G$, and $e$ is identified with the identity element.  Using the lifting property for connections gives us maps
\begin{eqnarray*}
\Omega_b(M) &\rightarrow & G \\
P_b(M) &\rightarrow & P.
\end{eqnarray*}
The map $\Omega_b(M) \rightarrow G$ is often referred to as the holonomy map.  

\begin{lem}
Let $G \rightarrow P \rightarrow M$ be a principal bundle with connection and $\Omega_b(M) \rightarrow P_b(M) \rightarrow M$ be the path space fibration.  The diagram
$$
\begin{CD}
P_b(M) @> >> P \\
@VVV              @VVV\\
M @> Id>> M
\end{CD}
$$
commutes.  Furthermore, the map $P_b(M)\rightarrow P$ commutes with the $\Omega_b(M)$ action on $P_b(M)$ and the $G$ action on $P$.
\end{lem}

\begin{proof}
This first part is the definition of lifting paths.  See (\cite{KN}, Proposition 3.2) for the second statement.
\end{proof}

This bundle map induces a map on the conjugate bundles $$
\begin{CD}
Conj(P_b(M)) @>>> Conj(P) \\
@VVV          @VVV\\
M @>>> M.
\end{CD}
$$  An element in $Conj(P_b(M))$ is represented by an element $(p_t, \alpha) \in P_b(M) \times \Omega_b(M)$. The induced map is defined by taking a representative $(p_t, \alpha)$ and sending it by the map $$(p_t, \alpha ) \mapsto [\widetilde {p_t}(1), \widetilde \alpha] \in Conj(P)= P \times_G G.$$  A loop $\alpha \in \Omega_b(M)$ lifts to a path $\widetilde \alpha$ starting at $e \in F_b(M)$ and ending in $F_b(M)$.  This path corresponds to an element in $G$.  A path $p_t \in P_b(M)$ lifts to a path $\widetilde p_t$ in $P$ starting at $e$.  Then $p_t \mapsto \widetilde p_t(1) \in P$.

\begin{prop}
Let $G \rightarrow P \rightarrow M$ be a principal $G$ bundle with connection and $\Omega_b(M) \rightarrow P_b(M) \rightarrow M$ be the path space fibration.  The map
\begin{eqnarray*}
Conj(P_b(M)) & \rightarrow&  Conj(P) \\
(p_t, \alpha )  &\mapsto&  (\widetilde p_t(1) , \widetilde \alpha )
\end{eqnarray*}
is well defined and independent of choice of basepoint $e \in F_b(M)$.
\end{prop}

\begin{proof}
Let $\beta \in \Omega_b(M)$.  For the map to be well defined, $(\widetilde {p_t \beta} , \widetilde{\beta \alpha \beta^{-1}})$ and $(\widetilde p_1, \widetilde \alpha)$ must be in the same equivalence class in $P \times_G G$.  We see that conjugating $(\widetilde p_1, \widetilde \alpha)$ by $\widetilde \beta \in G$ is $(\widetilde {p_t \beta} , \widetilde{\beta \alpha \beta^{-1}})$.  So the map is well defined.

Choosing a different point $e' \in P_b(M)$ changes the map $P_b(M) \rightarrow P$ by the $G$ action and changes the map $\Omega_b(M) \rightarrow G$ by a conjugation.  In the conjugate bundle, the images belong to the same equivalence class.
\end{proof}

Given a bundle $G \rightarrow P \rightarrow M$, with $G$ a connected Lie group, we look to construct a twisting cochain $\tau:H_*(M) \rightarrow H_*(G)$.  Then using the methods in Section \ref{section:algebraicsetting}, we obtain various structures on $H_*(M) \otimes H_*(G)$ modeling $H_*(P)$.  The twisting cochain will be in terms of the characteristic classes of the bundle.

\begin{prop}(\cite{DP}, p. $249$)
Let $G$ be a Lie group and $R$ a ring.  The cohomology, $H^*(BG;R)$ is a polynomial $R$-algebra of finite type on generators of even degree.
\end{prop}

For $H_*(BG)$, we need a separate argument.
\begin{lem}
Let $G$ be a connected Lie group.  Then $H_*(BG)$ is a free commutative algebra.
\end{lem}
\begin{proof}
The classifying space $BG$ is rationally equivalent to a product of Eilenberg Maclane spaces.  Furthermore, since $G$ is connected, the long exact sequence in homotopy groups of $G \rightarrow EG \rightarrow BG$, implies $\pi_1(BG)=0$.  The Eilenberg Maclane spaces here are then infinite loop spaces, and so $BG$ is rationally an infinite loop space.  This means $H_*(BG)$ is a Hopf algebra, which is commutative if the Lie bracket in $Prim(H_*(BG))$ is zero.  This bracket is equivalent to the Whitehead bracket on $\pi_*(Y)$ where $\Omega_b^2(Y)= BG$.  But $Y$ is a loop space, since $BG$ is, rationally, an infinite loop space.  And the Whitehead bracket on $H$-spaces is zero.

Hopf algebras are self dual, so $H^*(BG)$ is a Hopf algebra and $H_*(BG)$ is the dual Hopf algebra.  We see that $H_*(BG)$ is also a polynomial algebra.

\end{proof}

\subsection{Constructing the twisting cochain $H_*(M) \rightarrow H_*(G)$.}

The power series connection $\omega \in \Omega^*(M) \otimes \mathcal{L}(H_*(M)[-1])$, constructed in Section \ref{subsection:psc} will be used once more.  Theorem \ref{thm:psctocoalgebramap} defines a dg coalgebra map $T(H^*(M)[1]) \rightarrow T(\Omega^*(M)[1])$, which has an inverse $T(\Omega^*(M)[1]) \rightarrow T(H^*(M)[1]).$

Since $G$ is a connected Lie group, $H^*(BG)$ is a polynomial algebra.  This allows us to define maps from $H^*(BG)$ in terms of its polynomial generators.  Let $\{p_i \in H^*(M\})$ be the characteristic classes of a bundle $G \rightarrow P \rightarrow M$.  Then there is an algebra map $H^*(BG) \rightarrow \Omega^*(M)$ defined as follows.  Let $\{P_i \in H^*(BG)\}$ be the polynomial generators which pullback to the characteristic classes $\{p_i\}$.  Then define an algebra map by $P_i \mapsto \widehat p_i$, where $\widehat p_i \in \Omega^*(M)$ is a representative for $p_i$.  Extend the map as an algebra map to all of $H^*(BG)$.  The algebra map $H^*(BG) \rightarrow \Omega^*(M)$ defines a coalgebra map $T(H^*(BG)[1]) \rightarrow T(\Omega^*(M)[1]).$

So we have a coalgebra map $T(H^*(BG)[1]) \rightarrow T(\Omega^*(M) \rightarrow T(H^*(M)[1]),$ which defines an algebra map $T(H_*(M) [-1]) \rightarrow T(H_*(BG)[-1])$.  To this algebra map, there is a corresponding twisting cochain $H_*(M) \rightarrow T(H_*(BG)[-1])$.  Since $T(H_*(BG)[-1])$ is a model for $\Omega_b(BG)$, which is homotopy equivalent to $G$, we could do our work with twisting cochains now.

To replace $T(H^*(BG)[1] )$ with $H^*(G)$ we need to find a coalgebra map $H^*(G) \rightarrow T(H^*(BG)[1])$.  Recall that $H^*(G)$ is generated by odd dimensional generators $U_i$.  To each $U_i$ there is a generator of $H^*(BG)$ one degree higher, which we denote by $P_i$.  We define
\begin{eqnarray*}
f: H^*(G) &\rightarrow& T(H^*(BG)[1]) \\
U_i &\mapsto & P_i \\
U_{i_1} U_{i_2} &\mapsto & P_{i_1} \otimes P_{i_2} + P_{i_2} \otimes P_{i_1}
\end{eqnarray*}
and extending the map as an algebra map.  So $f(U_{i_1} \cdots U_{i_j}) = P_{i_1} \doublecup \cdots \doublecup P_{i_j}$, where $\doublecup$ is the shuffle product.

\begin{lem}
The map $f:H^*(G) \rightarrow T(H^*(BG)[1])$ is a map of differential graded coalgebras.  Therefore, $f$ is a map of differential graded Hopf algebras.
\end{lem}

\begin{proof}
The coproduct on $H^*(G)$ is given by
\begin{eqnarray*}
\Delta_{G}(U_{i_1} U_{i_2})=  U_{i_1} U_{i_2} \otimes 1 + U_{i_1} \otimes U_{i_2} + U_{i_2} \otimes U_{i_1} + 1 \otimes U_{i_1} U_{i_2},
\end{eqnarray*}
and extended so that $\Delta_G$ is an algebra map.  The coproduct on $T(H^*(BG)[1])$ is given by deconcatenation,
\begin{eqnarray*}
\Delta (P_{i_1} \otimes \cdots \otimes P_{i_k}) &=& \sum_j P_{i_1} \otimes \cdots P_{i_j} \bigotimes P_{i_{j+1}}
\otimes \cdots \otimes P_{i_k}.
\end{eqnarray*}

The following computation shows that  $f$ is a coalgebra map,
\begin{eqnarray*}
(f \otimes f) \circ \Delta(U_i U_j) &=& (f\otimes f) (U_i U_j \otimes 1 + U_i \otimes U_j + U_j \otimes U_i + 1 \otimes U_i U_j) \\
&=& (P_i \otimes P_j) \otimes 1 + (P_j \otimes P_i) \otimes 1 + P_i \otimes P_j  \\
&\,& + P_j \otimes P_i + 1 \otimes (P_i \otimes P_j) + 1 \otimes (P_j \otimes P_i) \\
&=& \Delta(P_i \otimes P_j + P_j \otimes P_i) \\
&=& \Delta f (U_i U_j).
\end{eqnarray*}

 The differential on $H^*(G)$ is zero, so for $f$ to be a chain map, $f$ must map to cocycles in $T(H^*(BG)[1])$.  We see that $\delta$ is zero on $P_i$.  Then since $f$ maps to shuffle products of $P_i$ and $\delta$ is a derivation with respect to the shuffle product, $f$ maps to cocycles.

 \end{proof}

To replace $T(H_*(BG)[-1])$ with $H_*(G)$, we take the dual of the above map to get a differential graded algebra map $T(H_*(BG)[-1])\rightarrow H_*(G)$.  So given a twisting cochain $\tau:H_*(M) \rightarrow T(H_*(BG)[-1])$, composing maps defines a twisting cochain $H_*(M) \rightarrow T(H_*(BG)[-1]) \rightarrow H_*(G)$.  Similarly, $H^*(G) \rightarrow H^*(M)$ is a twisting cochain obtained by composing the twisting cochain $T(H^*(BG)[1]) \rightarrow H^*(M)$ and the coalgebra map $H^*(G) \rightarrow T(H^*(BG)[1])$.

We summarize the construction of the twisting cochain and give a formula for it.  Let $G \rightarrow P \rightarrow M$ be a principal $G$ bundle, where $G$ is a connected Lie group and $M$ is a simply connected manifold.  Let $\{P_i\}$ be the multiplicative basis for $H^*(BG)$ where $p_i \in H^*(M)$ is the pullback of $P_i \in H^*(BG)$. The elements $P_i$ are even dimensional and correspond to an element $U_i \in H^*(G)$ such that $\{U_i\}$ form a basis for $H^*(G)$.  The following coalgebra maps are composed
\begin{enumerate}
\item{$T(\Omega^*M)[1]) \rightarrow T(H^*(M)[1])$ }
\item{$T(H^*(BG)[1]) \rightarrow T(\Omega^*(M)[1])$}
\item{$H^*(G) \rightarrow T(H^*(BG)[1])$}
\end{enumerate}
to define a coalgebra map $H^*(G) \rightarrow T(H^*(M)[1])$ which corresponds to a twisting cochain $H^*(G) \rightarrow H^*(M).$  When this process is carried out,  $\tau:H^*(G) \rightarrow H^*(M)$ is defined on generators by
 \begin{eqnarray*}
 H^*(G) &\rightarrow& H^*(M) \\
 U_i &\mapsto& p_i,
 \end{eqnarray*}
and zero on products of generators.

\begin{prop}

Consider the coalgebra structure on $H^*(G)$ given by group
multiplication and the $C_\infty$ algebra structure on $H^*(M)$
given by the cup product.  Then the map $\tau:H^*(G) \rightarrow
H^*(M)$ which on  generators is $ U_i \mapsto p_i$ and zero on
products of generators is the twisting cochain coming from the
twisting cochain $H_*(M) \hookrightarrow \mathcal{L}(H_*(M)[-1])$ given by the inclusion.

\end{prop}

\begin{proof}

There are no differentials on $H^*(G)$ and $H^*(M)$, and so it suffices to show that $m_n^{Hom}(\tau^{\otimes n})=0 $ for each $n$.  For $m_n^{Hom}(\tau^{\otimes n})$ to be possibly non-zero, we need to consider the product of $n$ generators $U_{i_1} \cdots U_{i_n}.$  We look at terms in $\Delta^n(U_{i_1} \cdots U_{i_n})$ of the form $$\sum U_{i_{\sigma(1)}} \otimes \cdots \otimes U_{i_{\sigma(n) }}.$$  Then  we apply $\tau$ to each factor and apply $m_n:H^*(M)^{\otimes n} \rightarrow H^*(M)$ of the $C_\infty$ algebra.  But each $m_n$ vanishes on shuffle products, so it is zero on products of these terms.

\end{proof}

\subsection{$A_\infty$ coalgebra of $H_*(M) \otimes_\tau H_*(G)$ for a principal $G$-bundle}

We can now define the twisted $A_\infty$ coalgebra structure on $H_*(M) \otimes H_*(G).$  We use the dual of $\tau:H^*(G) \rightarrow H^*(M)$, to get a twisting cochain.  The map is also denoted $\tau$ and is defined as
\begin{eqnarray*}
\tau:H_*(M) & \rightarrow & H_*(G) \\
p_i^* &\mapsto& U_i^*,
\end{eqnarray*}
is zero on products $p_{i_1}^* \cdots p_{i_n}^*$.  Note that $U_i^* \in Prim(G)$, and $[U^*_{i_1}, U^*_{i_2}]$ is defined.  The tensor differential on $H_*(M) \otimes H_*(G)$ is zero, so $\partial_\tau$ consists
only of twisted terms.  These terms are obtained by applying $\{c_n:H_*(M)\rightarrow H_*(M)^{\otimes n}\}$, applying $\tau$ to the last $n-1$ terms, bracketing the results, and then multiplying the resulting bracket with the element in $H_*(G)$.  The higher coproducts $c_2^\tau, c_3^\tau. \cdots $ are defined in the same way.

\begin{thm}
Let $\{ p\} $ be the characteristic classes of a $G$ bundle $G \rightarrow P \rightarrow M$, where $G$ is a connected Lie group and $M$ a simply connected manifold.  The maps $\{\partial_\tau, c_2^\tau, c_3^\tau, \cdots  \}$ define an $A_\infty$ coalgebra on $H_*(M) \otimes H_*(G)$ whose linear homology is isomorphic to $H_*(P)$.

\end{thm}

\begin{proof}

This is an application of Theorem \ref{thm:twistedcoalgebra}.

\end{proof}

The twisting cochain is more easily defined as $\tau:H^*(G) \rightarrow H^*(M)$, so the dual $A_\infty$ algebra can be made more explicit.  Note that if $C_*$ is a $C_\infty$ coalgebra, $H_*$ a Hopf algebra, and a twisting cochain $C_* \rightarrow H_*$ has its image in the primitives, then its dual map $\tau:H^* \rightarrow C^*$ has the property that $ker(\tau) \cup Prim(H_*) = H_*$.  This property of $\tau$ implies the derivation property dual to the statement that multiplying by a primitive element is a coderivation.  It is described in Figure \ref{figure:primitivederivation}.

\begin{figure}
\includegraphics[width=3.5in]{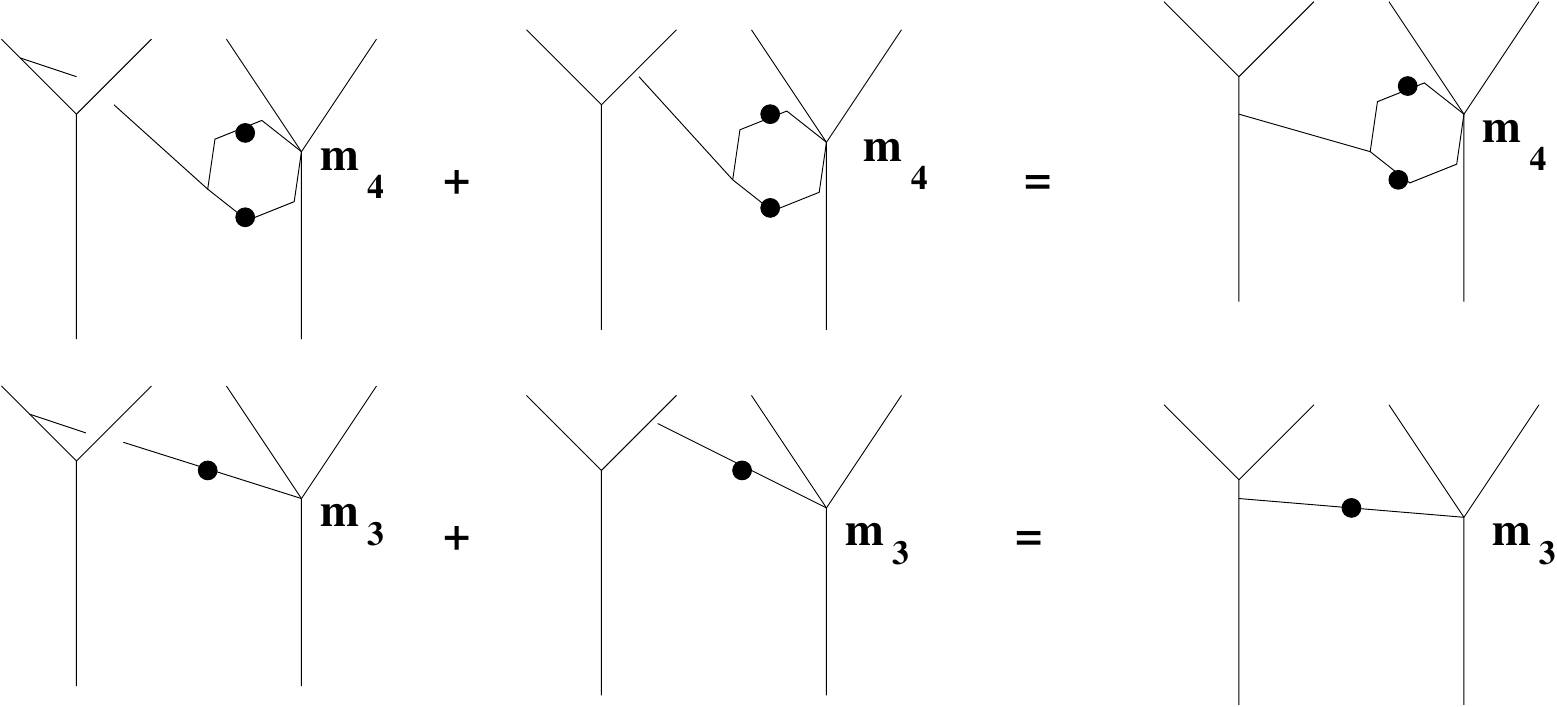}  \\

\caption{This identity is a consequence of the fact that $ker(\tau) \cup Prim(H_*) = H_*$.  The figure is dual to Figure \protect{ \ref{figure:primitivecoderivation}}.}
\label{figure:primitivederivation}
\end{figure}

We define an $A_\infty$ algebra on $H^*(G) \otimes H^*(M)$, where we view $H^*(G)$ as a Hopf algebra and $H^*(M)$ as a $C_\infty$ algebra.  The map $\partial_{\tau}:H^*(G) \otimes H^*(M) \rightarrow H^*(G) \otimes H^*(M)$ is given by
\begin{eqnarray*}
\partial_\tau(U_{i_1} \cdots U_{i_n} \otimes a) &=& \sum_{\sigma \in S_n} U_{i_{\sigma(1)}} \cdots U_{i_{\sigma(n-1)}} \otimes m_2(p_{i_{\sigma(n)}}\otimes a ) \\
&\,& +\sum_{\sigma \in S_n} U_{i_{\sigma(1)}} \cdots U_{i_{\sigma(n-2)}} \otimes m_3 (p_{i_{\sigma(n-1)}} \otimes  p_{i_{\sigma(n)}} \otimes a) \\
&\,& \vdots.
\end{eqnarray*}

The map $m_2^{\tau}:(H^*(G) \otimes H^*(M))^{\otimes 2} \rightarrow H^*(G) \otimes H^*(M)$ is given by
\begin{eqnarray*}
&\,&m_2^\tau(U_{i_1} \cdots U_{i_k} \otimes a , U_{i_{k+1}} \cdots U_{i_{n}} \otimes b)\\
 &=& U_{i_1} \cdots U_{i_n} \otimes m_2(a \otimes b) \\
&\,& +\sum_{\sigma \in S_n} U_{i_{\sigma(1)}} \cdots U_{i_{\sigma(n-1)}} \otimes m_3(p_{i_{\sigma(n)}} \otimes a \otimes b ) \\
&\,& +\sum_{\sigma \in S_n}  U_{i_{\sigma(1)}} \cdots U_{i_{\sigma(n-2)}} \otimes m_4(p_{i_{\sigma(n-1)}} \otimes p_{i_{\sigma(n)}} \otimes a \otimes b ) \\
&\,& \vdots.
\end{eqnarray*}

\begin{prop}
Let $\{p_i\}$ be the characteristic classes of a principal $G$ bundle $P \rightarrow M$, with $M$ simply connected and $G$ a connected Lie group.  The maps $\{ \partial_\tau, m_2^\tau, \cdots \}$ define an $A_\infty$ algebra on $H^*(G) \otimes H^*(M)$ whose linear cohomology is isomorphic to $H^*(P)$.
\end{prop}

\begin{proof}
This is the algebraic dual of Theorem \ref{thm:twistedcoalgebra}.  One can see that $\partial_\tau^2=0$ directly, as well.
\begin{eqnarray*}
&\,&\partial_\tau^2(U_{i_1} \cdots U_{i_n} \otimes a) \\
 &=& \sum_{\sigma'} \sum_{\sigma } U_{i_{\sigma' \sigma(1)}} \cdots U_{i_{\sigma'\sigma(n-2)}} \otimes m_2(p_{i_{\sigma' }} \otimes m_2(p_{i_{\sigma(n)}} \otimes a) ) \\
&\,&+ \sum_{\sigma'} \sum_{\sigma } U_{i_{\sigma' \sigma(1)}} \cdots U_{i_{\sigma'\sigma(n-3)}} \otimes m_3(p_{i_{\sigma'\sigma(n-2)}} \otimes p_{i_{\sigma' \sigma(n-1)}} \otimes m_2(p_{i_{\sigma(n)}} \otimes a) ) \\
&\,& \vdots \\
&\,& +\sum_{\sigma' } \sum_{\sigma } U_{i_{\sigma' \sigma(1)}} \cdots U_{i_{\sigma'\sigma(n-3)}} \otimes m_2(p_{i_{\sigma' (n-2)}} \otimes m_3(p_{i_{\sigma(n-1)}} \otimes p_{i_{\sigma(n)}} \otimes a) ) \\
&\,&+ \sum_{\sigma' } \sum_{\sigma } U_{i_{\sigma' \sigma(1)}} \cdots U_{i_{\sigma'\sigma(n-4)}} \otimes m_3(p_{i_{\sigma'\sigma(n-3)}} \otimes p_{i_{\sigma' \sigma(n-2)}} \otimes m_3(p_{i_{\sigma(n-1)}} \otimes p_{i_{\sigma(n)}} \otimes a) ) \\
&\,&\vdots.
\end{eqnarray*}
Note that on the $H^*(M)$ side of the tensor, there are compositions of $m_i$ and $m_j$.  The $C_\infty$ algebra relation on $H^*(M)$ states that such sums will be zero.

For the higher identities, we use the identity in Figure \ref{figure:primitivederivation} and follow the same argument that was made in Theorem \ref{thm:twistedcoalgebra}.
\end{proof}

\subsection{$A_\infty$ coalgebra on  $H_*(M) \otimes_\tau H_*(G)$ using conjugation action}

The conjugation action of $H_*(G)$ on itself is trivial when $G$ is a connected Lie group.  This shows that there is no twisting needed for the $A_\infty$ coalgebra on $H_*(M) \otimes H_*(G)$.  That is, the coalgebra  is given by $\{c_n \otimes \Delta^n\},$ where $\{c_n\}$ is the $C_\infty$ coalgebra  given by the diagonal map and $\Delta^n$ is the $n$-fold composition of the coproduct on $H_*(G)$.

\subsection{$A_\infty$ algebra on $H_*(M) \otimes_\tau H_*(G)$ using conjugation action}

Since the conjugation action is trivial, the $A_\infty$ algebra  on $H_*(M) \otimes H_*(G)$ is given by $\{m_n \otimes m_G\}$, with no twisting terms.  Here, $\{m_n\}$ is the $C_\infty$ algebra on $H_*(M)$ given by the intersection product and $m_G$ is the associative multiplication in $H_*(G)$.


\begin{thebibliography}{99999}

\bibitem[A]{A} L. Abrams, Two-dimensional topological quantum field theories and Frobenius algebras, \emph{J. Knot Theory Ramifications} 5 (1996), no.5, 569-587

\bibitem[B]{B} E.H. Brown Jr, {Twisted tensor products. I}, \emph{ Ann. of Math.} (2) 69 (1959) 223�246.

\bibitem[CL]{CL} J. Chuang, A. Lazarev, {$L_\infty$ maps and Twistings}, preprint,  arXiv:0912.1215

\bibitem[CS]{CS} M. Chas, D. Sullivan, {String Topology}, preprint,  math.GT/9911159

\bibitem[C]{C} K.-T. Chen, {Iterated integrals of differential forms and loop space homology}, \emph{Ann. of Math.} (2) 97 (1973), 217--246. 

\bibitem[CEG]{CEG} X. Chen, F. Eshmatov, W. Gan, {Quantization of the Lie Bialgebra of String Topology} \emph{Communications in Mathematical Physics}  
Volume 301, Number 1 (2011)


\bibitem[CG]{CG} X. Cheng, E. Getzler, { Transferring homotopy commutative algebraic structures},
\emph{J. Pure Appl. Algebra} 212 (2008), no. 11, 2535--2542.

\bibitem[DP]{DP} C. T. J. Dodson, P. Parker, A User's Guide to Algebraic Topology, \emph{Kluwer Academic Publishers} (1997)

\bibitem[DFN]{DFN} B.A. Dubrovin, A.T. Fomenko, S. P. Novikov, Modern Geometry- Methods and Applications Part II. Geometry and Topology of Manifolds,  \emph{Springer} (1995)

\bibitem[G]{G} W.L. Gan, Koszul Duality for Dioperads, \emph{Math. Res. Lett. } 10 (2003), no. 1, 109Ð124

\bibitem[GLS]{GLS} V.K.A.M. Gugenheim, L.A. Lambe, J.D. Stasheff, Algebraic Aspects of Chen's Twisting Cochain, \emph{Ill. J. of Math.} (1990) vol. 34, number 2

\bibitem[H1]{H1} R.M. Hain, Twisting cochains and duality between minimal algebras and minimal Lie algebras, \emph{Trans. Amer. Math. Soc.} 277 (1983) pp. 397-411

\bibitem[H2]{H2} R.M Hain, Iterated integrals and homotopy periods, \emph{Memoirs of the AMS} (1984)
 vol. 47, number 291

 \bibitem[Hep]{Hep} R. Hepworth, {String Topology for Lie Groups}, preprint arXiv:0905.1199v1
\bibitem[KN]{KN} S. Kobayashi, K. Nomizu, Foundations of Differential Geometry vol. I, \emph{Interscience Publishers} (1963)

\bibitem[HL]{HL} A. Hamilton, A. Lazarev,  {Homotopy Algebras and Noncommutative Geometry}. preprint 	arXiv:math/0410621v1 

\bibitem[Hue]{Hue} J. Huebschmann. {On the construction of $A_\infty$-structures}. Festschrift in honor of T. Kadeishvili�s 60th birthday, 2009. arXiv:0809.4791.

\bibitem[K]{K} T. Kadeishvili, {On the homology theory of fibre spaces} (Russian). Uspekhi Mat. Nauk. 35:3 (1980), 183�188, Russian Math. Surveys 35:3 (1980), 231�238, English version: arXiv:math/0504437.

\bibitem[LM]{LM} T. Lada, M. Markl, Strongly Homotopy Lie Algebras, \emph{Communications in Algebra} (23) 6 (1995) 2147-2161

\bibitem[M]{M} L. Menichi, A Batalin-Vilkovisky Algebra Morhpism
from Double Loop spaces to Free Loops, preprint arXiv:0908.1883 ,
accepted by Trans. Amer. Math. Soc.

\bibitem[Mer]{Mer} S. Merkulov, Strong homotopy algebras of a Kahler manifold, \emph{Internat. Math. Res. Notices}  no. 3, (1999), 153--164. 

\bibitem[MS]{MS} J. Milnor, J. Stasheff, Characteristic Classes, \emph{Princeton University Press} (1974)

 \bibitem[PS]{PS} A.V. Phillips, D.A. Stone, A Topological Chern-Weil Theory, \emph{Memoirs of the AMS} (1993) vol. 105, number 504

\bibitem[Q]{Q} D. Quillen, Rational Homotopy Theory, \emph{Annals of Math}  90 (1969)   205�295

\bibitem[V]{V}  B. Vallette, A Koszul duality for PROPs, \emph{ Trans. Amer. Math. Soc.} 359 (2007), no. 10, 4865Ð4943.
 \end{thebibliography}
\end{document}